\documentclass[twoside,a4paper,reqno,11pt]{amsart} 
\usepackage[top=28mm,right=30mm,bottom=28mm,left=30mm]{geometry}
\usepackage{amsfonts, amsmath, amssymb, mathrsfs, bm, latexsym, stmaryrd, array, mathtools, bold-extra}

\usepackage{color}
\usepackage{xcolor}
\usepackage[]{hyperref}
\hypersetup{
     colorlinks   = false,
     citecolor    = blue
}

\usepackage{bbm}
\usepackage[all,cmtip]{xy}

\usepackage{enumerate}
\usepackage[shortlabels]{enumitem}
\newtheorem{theorem}{Theorem}[section]

\newtheorem{lemma}[theorem]{Lemma}
\newtheorem{prop}[theorem]{Proposition}

\theoremstyle{definition}

\newtheorem*{notation}{Notation}

\newtheorem{prob}[theorem]{Problem}

\newcommand{\LL}{{\mathrm {L}}}
\newcommand{\SSS}{{\mathrm {S}}}
\newcommand{\OO}{{\mathrm {O}}}

\newcommand{\la}{\langle}
\newcommand{\ra}{\rangle}
\newcommand{\Z}{\mathbb{Z}}

\newcommand{\F}{\mathbb{F}}

\newcommand{\Alt}{\mathrm{A}}
\newcommand{\Sym}{\mathrm{S}}
\newcommand{\Irr}{{\mathrm {Irr}}}
\newcommand{\cd}{{\mathrm {cd}}}
\newcommand{\St}{{\mathrm {St}}}

\newcommand{\Centralizer}{{\mathbf {C}}}

\newcommand{\kernel}{{\mathrm {ker}}}

\providecommand{\Aut}{\mathop{\rm Aut}\nolimits}%
\providecommand{\Out}{\mathop{\rm Out}\nolimits}%
\begin{document}


\title[Simple exceptional groups of Lie type]{\bf Huppert's conjecture for finite simple exceptional groups of Lie type}

\author[H. P. Tong-Viet]{Hung P. Tong-Viet}
\address{Department of Mathematics and Statistics, Binghamton University, Binghamton, NY 13902-6000, USA}
\email{htongvie@binghamton.edu}

\date{\today}

\begin{abstract} Let $G$ be a finite group and let $\cd(G)$ be
the set of its complex irreducible character degrees. We show that if $\cd(G)=\cd(H),$ where $H$
is a finite simple exceptional group of Lie type, then $G\cong
H\times A,$ where $A$ is an abelian group. This completes the verification of Huppert's
Conjecture for all finite simple exceptional groups of Lie type.

\end{abstract}

\thanks{}
\maketitle

\section{Introduction}\label{sec:intro}
Let $G$ be a finite group and let $\Irr(G)$ be the set of its complex irreducible characters.  By a `character degree' we mean the degree of a complex irreducible character. Let $\cd(G)$ be the set of all character degrees of $G$. One of the main questions in the character theory of finite groups is to determine which properties of a finite group can be determined by its character table. Note that the character degree set of a group can be read off from the character table. Thus it is interesting to ask how much  the character degree set of a group knows about the group itself.

We cannot expect too much since the character degree set of a finite group $G$ and the character degree set of a direct product $G\times A$, where $A$ is an abelian group, coincide. Navarro \cite{Navarro1} showed that the character degree set alone cannot determine the solvability of the group.  He constructed a finite perfect group $H$ and a finite solvable group $G$ such that $\cd(G)=\cd(H)$. Furthermore, Navarro and Rizo \cite{NR} found a finite perfect group and a finite nilpotent group with the same character degree set.

On the other hand, Huppert proposed in \cite{Hupp} that all finite nonabelian simple groups are essentially determined by the set of their character degrees: if $H$ is a finite nonabelian simple group and $G$ is a finite group such that $\cd(G)=\cd(H)$, then $G\cong H\times A$ for some finite abelian group $A$. The main purpose of this paper is to prove the following theorem.

\begin{theorem}\label{th:main} Let $G$ be a finite group and let ${H}$ be a finite simple exceptional group of Lie type.  If
$\cd(G)= \cd({H}),$ then $G\cong H\times A,$ where $A$ is abelian.
\end{theorem}

Huppert's Conjecture  has been verified for various   finite simple exceptional groups of Lie type (see \cite{Hupp,Tong,HT1,HT2,Wake}). To complete the verification of this conjecture for finite simple exceptional groups of Lie type, we verify the conjecture for the following finite simple groups: 
$\textrm{F}_4(q),~{}^2\textrm{E}_6(q),~\textrm{E}_6(q),~\textrm{E}_7(q),~\textrm{E}_8(q) $ where $q$ is a prime power.

Huppert's Conjecture has been verified for all sporadic simple groups (see \cite{HT3} and the references therein), the alternating groups  \cite{BTZ} and many finite simple classical groups of small Lie rank (see \cite{HTW}). The conjecture remains open for finite simple classical groups of large rank.

In \cite{HMTW}, we extend the conjecture to include certain  quasi-simple groups.  Recently, an analog of Huppert's Conjecture was proposed by Qian in \cite[Problem 20.79]{KM} using codegrees. Recall that the codegree of a character $\chi\in\Irr(G)$ is the ratio $|G:\kernel(\chi)|/\chi(1)$, where $\kernel(\chi)$ denotes the kernel of $\chi$. Let $\textrm{cod}(G)$ be the set of all codegrees of $G$. Qian conjectured that if $G$ is a finite group and $H$ is a finite nonabelian simple group such that $\textrm{cod}(G)=\textrm{cod}(H)$, then $G\cong H.$ This conjecture has been verified for alternating groups, all sporadic simple groups, all finite simple exceptional groups of Lie type and finite simple projective special linear groups (see \cite{Tong24} and the references therein).

\medskip
We now describe our approach to the proof of Huppert's Conjecture. In \cite{Hupp}, Huppert proposed a $5$-step strategy to verify his conjecture and this has been used successfully to verify the conjecture for many finite simple groups. In \cite{BTZ}, we modified his method to verify the conjecture for alternating groups. We use this method and the data from \cite{Lubweb} to prove the main theorem.

Suppose that $G$ is a finite group and $H$ is a finite nonabelian simple group such that $\cd(G)=\cd(H)$. To verify Huppert's Conjecture for  $H$, we need to prove the following.

\smallskip
{\bf Step~1:} Show that $G'=G''$.

\smallskip
{\bf Step~2:}  If $G'/M$ is a chief factor of $G$,  then $G'/M\cong H$.

\smallskip
{\bf Step~3:}  Show that $M=1$.

\smallskip
{\bf Step~4:} Show that $G\cong G'\times C_G(G')$.

\smallskip

\smallskip
To verify Steps~$1$ and $2$, we follow the argument developed by Huppert. For Step~3, we  can reduce to the case when $L$ is a quotient of $G'$ and $N$ is a minimal normal elementary abelian subgroup of $L$ such that $L/N\cong H$ and  either $\Centralizer_L(N)=N$ or $\Centralizer_L(N)=L$ (see the proof of Proposition \ref{prop:step3}).  For the first possibility, we use a result due to Guralnick and Tiep \cite{Guralnick} on the non-coprime $k(GV)$-problem to obtain a contradiction. If the latter holds, then $L$ is a quasi-simple group.  This and Step 4 can now be dealt with using the data from \cite{Lubweb}.

The paper is organized as follows. We collect some results needed to verify Steps $1,2$ and $3$ in Section \ref{sec2}. In Section \ref{sec3}, we present all the information about the character degrees of the finite simple exceptional groups of Lie type that we need for the proof of the main theorem. We  verify Steps 1--4  in Sections \ref{sec4}--\ref{sec7}, respectively.

\begin{notation} If $n$ is an integer, then we denote by
$\pi(n)$ the set of all prime divisors of $n.$ If $G$ is a group, then we
write $\pi(G)$ to denote the set of all
prime divisors of the order of $G.$ If $p$ is a prime and $n>1$ is an integer, then we write $n_p$ for the largest power of $p$ that divides $n$. Let $\rho(G)=\cup_{\chi\in
\Irr(G)}\pi(\chi(1))$ be the set of all primes which divide
some irreducible character degrees of $G.$ The largest character
degree of $G$ is denoted by $b(G).$ For a group $G,$ let $k(G)$
be the number of conjugacy classes of $G.$
 If $N\unlhd G$ and $\theta\in
\Irr(N),$ then the inertia group of $\theta$ in $G$ is
denoted by $I_G(\theta).$ The set of all irreducible constituents of
$\theta^G$ is denoted by $\Irr(G|\theta).$ Finally, we
denote by $\Phi_n=\Phi_n(q)$ the cyclotomic polynomial in variable
$q.$ Other notation is standard.
\end{notation}

\section{Preliminaries}\label{sec2}

We collect various results that are needed for the proofs of the main theorem.
We start with the following lemmas  which will be used to verify Steps $2$ and $3.$

\begin{lemma}\label{lem5} If $S$ is
a nonabelian simple group, then there exists a nontrivial irreducible character $\theta$ of $S$
that extends to ${\Aut}(S).$ Moreover the following hold:

$(i)$ If $S$ is an alternating group of degree at least $7,$ then $S$ has two
characters of consecutive degrees $n(n-3)/2$ and $(n-1)(n-2)/2$ that extend to ${\Aut}(S).$

$(ii)$ If $S$ is a sporadic simple group or the Tits group, then $S$ has two nontrivial irreducible
characters which   extend to ${\Aut}(S)$ and have coprime degrees.  These characters are listed in Table \ref{Tab2}, 

$(iii)$ If $S$ is a simple group of Lie type then the Steinberg character $\St_S$ of $S$ of degree
$|S|_p$ extends to ${\Aut}(S).$
\end{lemma}

\begin{proof} This is essentially \cite[Theorems~$2,3,4$]{Bianchi}. However for sporadic simple groups and the Tits
group, we choose different pairs of coprime degrees from those in \cite{Bianchi}. 
\end{proof}
Lemma \ref{lem5} will be used together with the following extension result.
\begin{lemma}\label{lem:ext}
Let $N$ be a minimal normal subgroup of a finite group $G$ such  that $N\cong S^k,$ where $S$ is a nonabelian simple group. If $\theta\in \Irr(S)$ extends to ${\rm{Aut}}(S),$ then $\theta^k\in \Irr(N)$ extends to $G.$
\end{lemma}

\begin{proof}
This is Lemma $5$ in \cite{Bianchi}.
\end{proof}

We also need the following number theoretic result. It is used to show that a certain character degree of the simple
exceptional groups of Lie type cannot be written as a proper nontrivial power.

\begin{lemma}\label{lem7} If $x,y,m\ge 2$ are integers and $x$ is a prime power, then the equation $x^2+x+1=y^m$ has no
 solution.
\end{lemma}

\begin{proof} 
The equation in the lemma can be rewritten as $(x^3-1)/(x-1)=y^m$. This is a special case of the Nagell-Ljunggren equation $(x^n-1)/(x-1)=y^m.$  Nagell and Ljunggren show (see \cite{BM})  that if $n=3$, then the only solution of the equation with $x,y,m\ge 2$  is $(x,y)=(18,7)$. However, $x=18$ is not a prime power.
\end{proof}

We next list the irreducible characters of prime power degrees of the finite nonabelian  simple exceptional groups of Lie type. This is part of \cite[Theorem 1.1]{Malle}.

\begin{lemma}\label{lem:pp}
Let $G$ be one of the finite simple exceptional groups of Lie type listed in \eqref{eqn1} and let $\chi$ be a faithful irreducible character of $G$. Suppose that $\chi(1)=r^d$, where $r$ is a prime. Then $G$ is of characteristic $r$ and $\chi$ is the Steinberg character of $G$, so $\chi(1)=|G|_r$.
\end{lemma}
The next result is crucial to our proof of Step $3.$ Unfortunately,
we are unable to find a simple proof of this fact. The proof given
below is an easy consequence of a deep result due to Guralnick and
Tiep \cite{Guralnick} concerning the non-coprime
$k(GV)$-problem. For a finite group $G,$ let $b(G)$ be the largest
character degree of $G.$ It follows from character theory that
$|G|=\sum_{\chi\in\Irr(G)}\chi(1)^2$ and $k(G)=|\Irr(G)|,$ where
$k(G)$ is the number of conjugacy classes of $G.$ Thus $|G|\leq
k(G)b(G)^2$ and $b(G)^2\leq |G|.$

\begin{lemma}\label{lem8}
Let $L$ be a finite group and let $M$ be a minimal normal elementary abelian subgroup of $L.$ Assume that
$L/M$ is isomorphic to a simple exceptional group of  Lie type of Lie rank at least $4$  and that $C_L(M)=M.$ Then $b(L)>b(L/M).$
\end{lemma}

\begin{proof}

Let $H\cong L/M$ be a finite simple exceptional group of Lie type of Lie rank at least $4.$ By way of
contradiction, assume that $b(L)\leq b(L/M)=b(H).$  Since $M$ is an abelian normal subgroup of $L$, the conjugation action of $L$ on $M$ induces an action of $H$ on $M$. Using this action, we form the semidirect product $HM$.  By \cite[Lemma~2.4]{Guralnick}, $k(L)\leq k(HM).$ Now
$$|L|=\sum_{\chi\in\Irr(L)}\chi(1)^2\leq k(L)b(L)^2\leq k(HM)b(H)^2. $$
As $H$ is a finite simple exceptional group of Lie type of Lie rank at least $4,$ and $H$ acts irreducibly on $M$. By  \cite[Theorems~1.4,~1.5]{Guralnick},  $k(HM)\leq |M|/2.$
Moreover, $|L|=|H|\cdot |M|$ and so the above inequality becomes
$$|L|=|H|\cdot |M|\leq |M|b(H)^2/2. $$
Hence $|H|\leq b(H)^2/2<b(H)^2,$ a contradiction. Thus  $b(L)>b(L/M)$ as wanted.
\end{proof}

To verify Step 3 for finite simple classical groups, we need  the following result.
\begin{prob}
Let $G$ be a finite perfect group and let $M$ be a minimal normal elementary abelian subgroup of $G$ such that $G/M$ is a finite simple classical group. If $C_G(M)=M$, then $G$ has an irreducible character $\chi$ such that $\chi(1)$ divides no character degree of $G/M.$
\end{prob}
It is possible that $b(G)>b(G/M)$ in the above situation. To finish Step $3$, we need to show that if $L$ is a finite quasi-simple group with $L/Z(L)$ a finite simple classical group, then some character degree of $L$ divides no character degree of $L/Z(L)$.

Finally, we need the following result due to Zsigmondy.
\begin{lemma}\emph{(\cite[Theorems ~5.2.14,~5.2.15]{KL}).}\label{Zsigmondy} Let $q$ and $n$ be integers with $q\geq 2,$ and $n\geq 3.$
Suppose that $(q,n)\neq (2,6).$ Then $q^n-1$ has a prime divisor
$\ell$ such that $\ell\equiv 1$ \emph{(mod $n$)} and $\ell$ does not
divide $q^k-1$ for any $k<n.$ Moreover if $\ell\mid q^k-1$ for some
integer $k,$ then $n\mid k.$
\end{lemma}

Such a prime $\ell$ is a \emph{primitive prime divisor} (ppd for short) and
is denoted by $\ell_n(q).$ If $n$ is odd and $(q,n)\neq (2,3)$
then $q^{2n}-1$ has a primitive prime divisor denoted by $\ell_{-n}(q).$ In terms of the
cyclotomic polynomials, as $q^n-1=\prod_{k\mid n}\Phi_k,$ if $\ell$ is a primitive prime divisor of $q^n-1,$ then
$\ell\mid \Phi_n$ but $\ell\nmid \Phi_k$ for any $k<n.$ In this
situation, we also say that $\ell$ is a primitive prime divisor of
$\Phi_n.$


\begin{table}[h]
 \begin{center}
  \caption{Character degrees of sporadic simple groups and the Tits group} \label{Tab2}
  \begin{tabular}{lll|llr}
   \hline
   Group  & Character &Degree&Group&Character&Degree\\ \hline
   $\textrm{M}_{11}$&$\chi_8$&$2^2\cdot 11$&$\textrm{O'N}$&$\chi_2$&$2^6\cdot 3^2\cdot 19$\\
   &$\chi_9$&$3^2\cdot 5$&&$\chi_{19}$&$7^3\cdot 11\cdot 31$\\

   $\textrm{M}_{12}$&$\chi_7$&$2\cdot 3^3$&$\textrm{Co}_3$&$\chi_3$&$11\cdot 23$\\
   &$\chi_8$&$5\cdot 11$&&$\chi_6$&$2^7\cdot 7$\\

   $\textrm{J}_1$&$\chi_5$&$2^2\cdot 19$&$\textrm{Co}_2$&$\chi_3$&$11\cdot 23$\\
   &$\chi_6$&$7\cdot 11$&&$\chi_{22}$&$3^6\cdot 5^3$\\

   $\textrm{M}_{22}$&$\chi_2$&$3\cdot 7$&$\textrm{Fi}_{22}$&$\chi_{56}$&$2^{17}\cdot 11$\\
   &$\chi_5$&$5\cdot 11$&&$\chi_{57}$&$3^9\cdot 7\cdot 13$\\

   $\textrm{J}_2$&$\chi_6$&$2^2\cdot 3^2$&$\textrm{HN}$&$\chi_{10}$&$3^4\cdot 11\cdot 19$\\
   &$\chi_{13}$&$5^2\cdot 7$&&$\chi_{45}$&$2^{10}\cdot 5^5$\\

   $\textrm{M}_{23}$&$\chi_3$&$3^2\cdot 5$&$\textrm{Ly}$&$\chi_7$&$2^8\cdot 7\cdot 67$\\
   &$\chi_9$&$11\cdot 23$&&$\chi_{50}$&$3\cdot 5^6\cdot 31\cdot 37$\\

   $\textrm{HS}$&$\chi_2$&$2\cdot 11$&$\textrm{Th}$&$\chi_2$&$2^3\cdot 31$\\
   &$\chi_7$&$5^2\cdot 7$&&$\chi_7$&$5^3\cdot 13\cdot 19$\\

   $\textrm{J}_3$&$\chi_6$&$2^2\cdot 3^4$&$\textrm{Fi}_{23}$&$\chi_4$&$13\cdot 17\cdot 23$\\
   &$\chi_{13}$&$5\cdot 17\cdot 19$&&$\chi_{94}$&$2^{18}\cdot 5^2\cdot 7\cdot 11$\\

   $\textrm{M}_{24}$&$\chi_7$&$2^2\cdot 3^2\cdot 7$&$\textrm{Co}_1$&$\chi_3$&$13\cdot 23$\\
   &$\chi_8$&$11\cdot 23$&&$\chi_{17}$&$2\cdot 5^4\cdot 7^2\cdot 11$\\

   $\textrm{McL}$&$\chi_2$&$2\cdot 11$&$\textrm{J}_4$&$\chi_2$&$31\cdot 43$\\
   &$\chi_{14}$&$3^6\cdot 7$&&$\chi_{11}$&$2^3\cdot 3^2\cdot 23\cdot 29\cdot 37$\\

   $\textrm{He}$&$\chi_9$&$3\cdot 5^2\cdot 17$&$\textrm{Fi}_{24}'$&$\chi_2$&$13\cdot 23\cdot 29$\\
   &$\chi_{15}$&$2^7\cdot 7^2$&&$\chi_6$&$5^2\cdot 7^3\cdot 11\cdot 17$\\

   $\textrm{Ru}$&$\chi_5$&$3^3\cdot 29$&$\textrm{B}$&$\chi_2$&$3\cdot 31\cdot 47$\\
   &$\chi_{20}$&$2^2\cdot 5^3\cdot 7\cdot 13$&&$\chi_{119}$&$2^{39}\cdot 11\cdot 19\cdot 23$\\

   $\textrm{Suz}$&$\chi_2$&$11\cdot 13$&$\textrm{M}$&$\chi_2$&$47\cdot 59\cdot 71$\\
   &$\chi_{43}$&$2^{10}\cdot 3^5$&&$\chi_{16}$&$5^9\cdot 7^6\cdot 11^2\cdot 17\cdot 19$\\

   ${}^2\textrm{F}_4(2)'$&$\chi_8$&$5^2\cdot 13$&$$&$$&$$\\
   &$\chi_{20}$&$2^6\cdot 3^3$&&$$&$$\\
   \hline
\end{tabular}
\end{center}
\end{table}


\section{Character degrees of simple exceptional groups of Lie type}\label{sec3}

We now collect the results concerning the character degrees of
exceptional groups of Lie type that we need for the proof of Theorem \ref{th:main}. We first
consider the following setup (see, for example, \cite{MTest}).  Let $\mathcal{G}$ be a simple simply connected algebraic group  defined
over the algebraic closure of a finite field ${\F}_p$ of characteristic $p$ and let $F:
\mathcal{G}\longrightarrow \mathcal{G}$ be a Steinberg endomorphism such that the fixed point group $L=\mathcal{G}^F$ is
finite and $H\cong L/Z(L)$  is a finite simple group. Let $\mathcal{G}^*$ be a group dual to $\mathcal{G}$
with a dual Frobenius map $F^*$ and let $H^*=({\mathcal{G}^*})^{F^*}.$ 
We use the notation $H_{sc}$ for $L$ and $H_{ad}$ for $H^*.$

According to Deligne - Lusztig theory  (see  \cite{Car85,DM,GM}), there is a
natural partition $$\Irr(L)=\bigcup_{\substack{s\in H^*/\sim\\s \:\text{semisimple}}}
\mathcal{E}(L,s)$$ of $\Irr(L),$ where the union runs over $H^*$-conjugacy classes of semisimple
elements. The irreducible characters of $L$ in $\mathcal{E}(L,1)$ are {\em unipotent
characters} of $L.$ The unipotent character degrees of simple exceptional groups of Lie type are
given explicitly in \cite[13.9]{Car85}.  Moreover for each $s,$ there is a bijection map $$\psi_s:
\mathcal{E}(L,s)\longrightarrow \mathcal{E}(\Centralizer_{H^*}(s),1)$$ such that for each
$\chi\in\mathcal{E}(L,s),$ the degree of $\chi$ is given by
$$\chi(1)=|H^*:\Centralizer_{H^*}(s)|\psi_s(\chi)(1)$$ (see for example \cite{Malle99}). The {\em semisimple
character} corresponding to the conjugacy class with representative $s$ is the character $\chi_s$
such that $\psi_s(\chi_s)$ is the principal character of $\Centralizer_{H^*}(s).$ In this case
$\chi_s(1)=|H^*:\Centralizer_{H^*}(s)|.$

By results of Lusztig (see \cite{Malle08}), we know that all the unipotent characters of $L$ are trivial on the center
of $L$; therefore,   they are exactly all the unipotent characters of $H\cong L/Z(L).$ Moreover, unipotent characters of $H$ extend to $\Aut(H),$ apart
from some exceptions (see \cite[Theorem~2.5]{Malle08}). This fact will be used frequently in the
proofs of Steps $2$ and $3.$
Some results on semisimple characters of $L$ can be found in \cite[Lemma~2.5]{Dolfi}.
In Table \ref{Tab1}, we list the $p$-parts of the degrees of some unipotent characters of finite simple groups of Lie type in characteristic $p$. Note that we follow the ATLAS notation  \cite{ATLAS} for finite simple groups of Lie type.

\begin{table}
 \begin{center}
  \caption{Some unipotent characters of finite simple groups of Lie type} \label{Tab1}
  \begin{tabular}{llr}
   \hline
   $S=S(p^b)$  & Symbol &$p$-part of degree\\ \hline
   $\textrm{L}_n^\epsilon(p^b),n\geq 3$ & $(1^{n-2},2)$&$p^{b(n-1)(n-2)/2}$\\
   $\textrm{S}_{2n}(p^b),p=2$&$\binom{0\:1\:2\:\cdots\:n-2\:n-1\:n}{\:\:1\:2\cdots\:n-2}$&$2^{b(n-1)^2-1}$\\

   $\textrm{S}_{2n}(p^b),p>2$ && $p^{b(n-1)^2}$\\
   $\textrm{O}_{2n+1}(p^b),p>2$  &$\binom{0\:1\:2\:\cdots\:n-2\:n-1\:n}{\:\:1\:2\cdots\:n-2}$& $p^{b(n-1)^2}$\\
   $\textrm{O}_{2n}^+(p^b)$&$\binom{0\:1\:2\:\cdots\:n-3\:n-1}{\:1\:2\:3\cdots\:n-2\:n-1}$&$p^{b(n^2-3n+3)}$\\
   $\textrm{O}_{2n}^-(p^b)$&$\binom{0\:1\:2\cdots\:n-2\:n}{1\:2\cdots\:n-2}$&$p^{b(n^2-3n+2)}$\\
   ${}^3\textrm{D}_4(p^b)$&$\phi_{1,3}''$&$p^{7b}$\\
   $\textrm{F}_4(p^b)$&$\phi_{9,10}$&$p^{10b}$\\
   ${}^2\textrm{F}_4(q^2)$&${}^2B_2[a],\epsilon$&$\frac{1}{\sqrt{2}}q^{13}$\\
   $\textrm{E}_6(p^b)$&$\phi_{6,25}$&$p^{25b}$\\
   ${}^2\textrm{E}_6(p^b)$&$\phi_{2,16}''$&$p^{25b}$\\
   $\textrm{E}_7(p^b)$&$\phi_{7,46}$&$p^{46b}$\\
   $\textrm{E}_8(p^b)$&$\phi_{8,91}$&$p^{91b}$\\\hline
\end{tabular}
\end{center}
\end{table}

The results in this section are based on the data given in \cite{Lubweb}. 
From this, we
can extract all the information about the character degrees of both $L$ and $H^*,$ the simply
connected and adjoint versions of the simple exceptional group of Lie type $H.$ Note that if
$d:=|Z(L)|=|H^*:H|$ is trivial, then all these versions of $H$ are isomorphic and we obtain
the character degrees of the simple group $H.$ Otherwise, we only know that $\cd(H)\subseteq
\cd(L)$ and every character degree of $H$ divides some degree of $H^*.$

We are now ready to list the results on the character degree sets of finite simple exceptional groups of
Lie type.  From now on we assume that $q=p^f$ is a power of a prime $p,$ where
$f\geq 1$ and $H$ is isomorphic to one of the following simple groups:
\begin{equation}\label{eqn1}
\textrm{F}_4(q),~{}^2\textrm{E}_6(q),~\textrm{E}_6(q),~\textrm{E}_7(q),~\textrm{E}_8(q).
\end{equation} 
Since  $H\cong \textrm{F}_4(2)$ has been
handled in \cite{HT}, we assume that $q\ge 3$ when $H\cong \textrm{F}_4(q).$ As a demonstration, we  give a detailed proof for the case  $H\cong {}^2{\rm E}_6(q)$. The remaining cases are similar and we skip their proofs.


\begin{lemma}\label{2E6} Let $H$ be the simple exceptional group ${}^2{\rm E}_6(q)$ and let $\ell_i,i=1,2,3,4$ be
 primitive prime divisors of $\Phi_{18},\Phi_{12},\Phi_8$ and $\Phi_{10},$ respectively. Let $x\neq q^{36}$
 be a nontrivial character degree of $L\cong {}^2{\rm E}_6(q)_{sc}.$ Then the following hold.

\begin{enumerate}[$(i)$]

\item If $(\ell_1\ell_2,x)=1,$ then $x$ is the degree of unipotent
characters labeled by the symbols ${}^2E_6[\theta^i],i=1,2$  with
degree
$\frac{1}{3}q^{7}\Phi_1^4\Phi_2^6\Phi_4^2\Phi_8\Phi_{10}.$

\item If $(\ell_2\ell_3,x)=1,$ then $x$ is the degree of  the
unipotent characters labeled by the symbols $\phi_{8,3}'$ or
$\phi_{8,9}''$ with degree
$\frac{1}{2}q^{3}\Phi_2^4\Phi_6^2\Phi_{10}\Phi_{18}
~\mbox{or~}~\frac{1}{2}q^{15}\Phi_2^4\Phi_6^2\Phi_{10}\Phi_{18}
.$

\item We have $(\ell_1\ell_3,x)>1$ and  $(p\ell_1\ell_4,x)>1.$

\item If $q=2$  then
$(p\ell_3\ell_4,x)>1,$ and if $q>2$ and $(p\ell_3\ell_4,x)=1,$ then $x=\Phi_3\Phi_6^2\Phi_{12}\Phi_{18}.$

\item $H$ possesses two semisimple irreducible characters
$\chi_i,i=1,2$ corresponding to the maximal tori with orders
$\Phi_{18}$ and $\Phi_{12}\Phi_6,$  with degrees
$$\Phi_1^4\Phi_2^6\Phi_3^2\Phi_4^2\Phi_6^3\Phi_8\Phi_{10}\Phi_{12},
\Phi_1^4\Phi_2^6\Phi_3^2\Phi_4^2\Phi_6^2\Phi_{8}\Phi_{10}\Phi_{18},$$
respectively, such that no proper multiple of each degree is a degree
of $H.$

\item The $p$-part of $x$ is at most $q^{25}.$

\item The unipotent character  labeled by the symbol $\phi_{2,4}'$
with degree $q\Phi_{8}\Phi_{18}$
 is the
smallest nontrivial character degree of $H.$

\item Assume $x,y\in \cd(L)-\{1,q^{36}\}.$ If $p\neq 3,$ then $(x,y)>1.$ If $p=3$ and $(x,y)=1,$
then $\{x,y\}=\{\frac{1}{3}q^{7}\Phi_1^4\Phi_2^6\Phi_4^2\Phi_8\Phi_{10},~\Phi_3\Phi_6^2\Phi_{12}\Phi_{18}\}$.

\item $L$ has no consecutive degrees.

\item Suppose that $3\mid q+1$. There exists
$\chi\in\Irr(L)$ of degree
$\frac{1}{3}q^{9}\Phi_1^3\Phi_2^4\Phi_3\Phi_4^2\Phi_6\Phi_8\Phi_{10}\Phi_{12}$
with multiplicity $6$ such that $\chi(1)$ divides no degree of $H.$
Similarly $H^*\cong {}^2{\rm E}_6(q)_{ad}$ has an irreducible character $\psi$  with
$\psi(1)=q^9\Phi_1\Phi_3^2\Phi_4^2\Phi_8\Phi_{10}\Phi_{12}\Phi_{18}$
 with multiplicity $1$ such that $\psi(1)$ is not a degree of $H.$

\end{enumerate}
\end{lemma}

\begin{proof} The lists of character degrees of ${}^2{\rm E}_6(q)_{sc}$ and
${}^2{\rm E}_6(q)_{ad}$ are available in \cite{Lubweb}. Statements
$(i)-(vi)$ can be checked easily. For $(vii)$, see \cite{Lubeck1}.  For $(viii),$ suppose that
$x,y\in \cd(L)-\{1,q^{36}\}$ and $(x,y)=1.$ Assume first that
$(\ell_1\ell_2,x)=1.$ By $(i),$ 
$x=\frac{1}{3}q^{7}\Phi_1^4\Phi_2^6\Phi_4^2\Phi_8\Phi_{10}$ and so
$p\ell_3\ell_4\mid x.$ Therefore,
$(p\ell_3\ell_4,y)=1.$ It follows from $(iv)$ that $ q>2$ and
$y=\Phi_3\Phi_6^2\Phi_{12}\Phi_{18}.$ It suffices to show that
$p=3.$ Assume that $p\neq 3.$ Then $q^2-1$ is divisible by $3.$ Now $\Phi_1^4\Phi_2^4=(q^2-1)^4$ is divisible by $3^4$ so
$x=\frac{1}{3}q^7\Phi_1^4\Phi_2^6\Phi_4^2\Phi_8\Phi_{10}$ is
divisible by $3.$ Also
$$\Phi_3\Phi_6=q^4+q^2+1=(q^4-1)+(q^2-1)+3=(q^2-1)(q^2+2)+3$$ is divisible by $3$ and thus $3$
divides both $x$ and $y$ and hence $(x,y)>1,$  a
contradiction. Thus $p=3$ as required.  Hence we
can assume that $(\ell_1\ell_2,x)>1.$ By interchanging the roles of
$x$ and $y,$ we can also assume that $(\ell_1\ell_2,y)>1.$ Without
loss of generality, we assume that $\ell_2\mid x.$ So,
$(\ell_2,y)=1$ hence $\ell_1\mid y$ and thus $(\ell_1,x)=1.$ By
$(iii),$  $\ell_3\mid x$ so $(\ell_3,y)=1$ and
hence $(\ell_2\ell_3,y)=1.$ By $(ii)$, 
$p\ell_1\ell_4\mid y$ and thus $(p\ell_1\ell_4,x)=1,$ which
contradicts $(iii)$. This proves claim $(viii)$.

For $(ix),$ without loss of generality, assume that
$x=y+1.$ By $(vii),$ we can assume that $x>y>1.$ But $q^{36}\pm 1$
do not divide $|L|,$  so $L$ has no such
degree; therefore, we assume that $x,y\in\cd(L)-\{1,q^{36}\}.$
Moreover $(x,y)=1$ and so  by $(viii),$ $p=3$ and
$$\{x,y\}=\{\frac{1}{3}q^7\Phi_1^4\Phi_2^6\Phi_4^2\Phi_8\Phi_{10},\Phi_3\Phi_6^2\Phi_{12}\Phi_{18}\}.$$
Since $q\geq 3,$ 
$$\frac{1}{3}q^7\Phi_1^4\Phi_2^6\Phi_4^2\Phi_8\Phi_{10}>q^{25}>q^{17}>\Phi_3\Phi_6^2\Phi_{12}\Phi_{18}$$
 and thus
$x=\frac{1}{3}q^7\Phi_1^4\Phi_2^6\Phi_4^2\Phi_8\Phi_{10}~\mbox{and}~
y=\Phi_3\Phi_6^2\Phi_{12}\Phi_{18}.$ But then
$$y+1<q^{17}+1<q^{18}<q^{25}<x$$ which is a contradiction as $x=y+1.$

For $(x),$ assume that $3\mid q+1.$ Using \cite{Lubweb},
${}^2{\rm E}_6(q)_{sc}$ possesses an irreducible character $\chi$ with
degree and multiplicity as in $(x).$ Also
$\chi(1)$ divides no other degree of $L.$ Since $\cd(H)\subseteq
\cd(L),$ $\chi(1)$ divides no degree of $H-\{\chi(1)\}.$ Hence it
suffices to show that $\chi(1)$ is not a degree of $H.$ Note that
$H$ is a normal subgroup of index $3$ in ${}^2{\rm E}_6(q)_{ad}$ and so,
for every irreducible character $\varphi$ of $H,$ either
$\varphi(1)$ or $3\varphi(1)$ is a degree of ${}^2{\rm E}_6(q)_{ad}$
depending on whether $\varphi$ is extendible to ${}^2{\rm E}_6(q)_{ad}$ or
not. Therefore, to show that $\chi(1)$ is not a degree of $H,$ we
only need to check that neither $\chi(1)$ nor $3\chi(1)$ is a degree
of ${}^2{\rm E}_6(q)_{ad}.$ Using \cite{Lubweb}, we check that
$\chi(1)$ satisfies these properties, which completes the proof of
our first claim. The second claim follows easily using
\cite{Lubweb}. The proof is now complete.
\end{proof}

\begin{lemma}\label{F4} Let $H\cong {\rm F}_4(q),$ where $q\geq 3,$ and
let $x\in \cd(H)$ with $x\not\in \{1,q^{24}\}$. Let $\ell_i,i=1,2,\cdots,4,$ be primitive prime divisors of
$\Phi_{12},\Phi_8,\Phi_6$ and $\Phi_3,$ respectively.  Then the following hold.

\begin{enumerate}[$(i)$]

\item If $(\ell_1\ell_2,x)=1,$ then
$x=\frac{1}{4}q^4\Phi_1^4\Phi_2^4\Phi_3^2\Phi_6^2.$

\item If   $(\ell_1,x)=1,$ then $p\mid x$ or
$x=\Phi_1^4\Phi_2^4\Phi_3^2\Phi_4^2\Phi_6^2\Phi_8.$

\item If   $(\ell_2,x)=1,$ then $p\mid x$ or $x$ is 
 $\Phi_3\Phi_6\Phi_{12}~(q ~odd)$ or $\Phi_1^4\Phi_2^4\Phi_3^2\Phi_4^2\Phi_6^2\Phi_{12}.$

\item If   $(\ell_3\ell_4,x)=1,$ then $x=\frac{1}{2}q\Phi_4\Phi_8\Phi_{12},$ $q^3\Phi_4^2\Phi_8\Phi_{12},$ $\frac{1}{3}q^4\Phi_1^4\Phi_2^4\Phi_4^2\Phi_{8},$ 
$q^9\Phi_4^2\Phi_{8}\Phi_{12},$ or $\frac{1}{2}q^{13}\Phi_4\Phi_8\Phi_{12}.$

\item There exist $\chi_i\in\Irr(H),i=1,2,$  with 
degrees
$\Phi_1^4\Phi_2^4\Phi_3^2\Phi_4^2\Phi_6^2\Phi_8$ and $\Phi_1^4\Phi_2^4\Phi_3^2\Phi_4^2\Phi_6^2\Phi_{12},$
respectively, such that no proper multiple of each degree is a
degree of $H.$

\item The $p$-part of $x$ is at most $q^{16}$ when $q$ is odd and
at most $\frac{1}{2}q^{13}$ when $q$ is even.

\item The smallest nontrivial character degree of $H$ is
$\Phi_3\Phi_6\Phi_{12}$ when $q$ is odd and 
$\frac{1}{2}q\Phi_1^2\Phi_3^2\Phi_8$ when $q$ is even.

\item Assume $x,y\in \cd(H)-\{1,q^{24}\}.$
If $p\neq 3,$ then $(x,y)>1.$
If $p= 3$ and $(x,y)=1,$ then
$\{x,y\}=\{\Phi_3\Phi_6\Phi_{12},\frac{1}{3}q^4\Phi_1^4\Phi_2^4\Phi_4^2\Phi_{8}\}.$

\item $H$ has no consecutive degrees.

\end{enumerate}
\end{lemma}

\begin{lemma}\label{E6} Let $H$ be the simple exceptional group ${\rm E}_6(q)$ and let $\ell_i,i=1,2,3,4$
be primitive prime divisors of $\Phi_{12},\Phi_{9},\Phi_8$ and $\Phi_{5},$
respectively.  Let $x\neq q^{36}$ be a nontrivial
character degree of $L.$  Then the following hold.

\begin{enumerate}[$(i)$]

\item If $(\ell_1\ell_2,x)=1,$ then $x$ is the degree of  unipotent characters labeled by the symbols
$E_6[\theta^i],i=1,2,$ with degree $\frac{1}{3}q^7\Phi_1^6\Phi_2^4\Phi_4^2\Phi_5\Phi_8.$

\item We have $(x,\ell_2\ell_3)>1.$

\item If $(x,\ell_1\ell_3)=1,$ then $x$ is the degree of unipotent characters labeled by the symbols $(D_4,1)$ or
$(D_4,\epsilon)$ with degrees $\frac{1}{2}q^3\Phi_1^4\Phi_3^2\Phi_5\Phi_9~\mbox{or~}\frac{1}{2}q^{15}\Phi_1^4\Phi_3^2\Phi_5\Phi_9.$

\item We have $(x,p\ell_2\ell_4)>1,$ and  if $(x,p\ell_3\ell_4)=1$
then $x=\Phi_3^2\Phi_6\Phi_9\Phi_{12}.$

\item $H$ possesses two irreducible characters $\chi_i,i=1,2$ with
degrees
$\Phi_1^6\Phi_2^4\Phi_3^2\Phi_4^2\Phi_6^2\Phi_5\Phi_{8}\Phi_{9}$ and $\Phi_1^6\Phi_2^4\Phi_3^3\Phi_4^2\Phi_6^2\Phi_{5}\Phi_{8}\Phi_{12},$
respectively such that no proper multiples of each degree is a degree of
$H.$

\item The $p$-part of $x$ is at most
$q^{25}.$

\item The unipotent character labeled by the symbol $\phi_{6,1}$ with degree $q\Phi_8\Phi_{9}$ is
 the smallest nontrivial character degree of $H.$

\item Assume $x,y\in \cd(L)-\{1,q^{36}\}.$ If $p\neq 3,$ then
$(x,y)>1.$ If $p= 3$ and $(x,y)=1,$ then
$\{x,y\}=\{\frac{1}{3}q^7\Phi_1^6\Phi_2^4\Phi_4^2\Phi_5\Phi_8,~\Phi_3^2\Phi_6\Phi_9\Phi_{12}\}.$

\item $L$ has no consecutive degrees.

\item Assume $3\mid q-1.$ There exists $\chi\in\Irr(L)$ of degree
$\frac{1}{3}q^{9}\Phi_1^4\Phi_2^3\Phi_3\Phi_4^2\Phi_5\Phi_6\Phi_{8}\Phi_{12}$ with multiplicity
$6$ such that $\chi(1)$ divides no degree of $H.$
 Moreover ${\rm E}_6(q)_{ad}$ has an irreducible character $\psi$ of degree $q^9\Phi_2\Phi_4^2\Phi_5\Phi_6^2\Phi_8\Phi_{9}\Phi_{12}$
 with multiplicity $1$ such that $\psi(1)$ is not a degree of $H.$

\end{enumerate}
\end{lemma}

\begin{lemma}\label{E7} Let $H$ be the simple exceptional group ${\rm E}_7(q)$ and let $\ell_i,i=1,2,\cdots,5$
be primitive prime divisors of $\Phi_{18},\Phi_{14},$ $\Phi_{12},\Phi_9,$ and $\Phi_7,$
respectively. Let $x$ be a nontrivial
character degree of $L$ with $x\neq q^{63}.$ Then the
following hold.

\begin{enumerate}[$(i)$]

\item If $(\ell_1\ell_2,x)=1,$ then $x$ is the degree of unipotent
characters labeled by the symbol $E_7[\xi]$ or $E_7[-\xi]$ with
degree
$\frac{1}{2}q^{11}\Phi_1^7\Phi_3^3\Phi_4^2\Phi_5\Phi_7\Phi_8\Phi_9\Phi_{12}.$

\item If $(\ell_2\ell_3,x)=1,$ then $x$ is the  degree of 
unipotent characters labeled by the symbols $(D_4,\epsilon_1)$ or
$(D_4,\epsilon_2)$ with degree
$$\frac{1}{2}q^{4}\Phi_1^4\Phi_3^2\Phi_5\Phi_7\Phi_9\Phi_{10}\Phi_{18}
\text{ or }\frac{1}{2}q^{25}\Phi_1^4\Phi_3^2\Phi_5\Phi_7\Phi_9\Phi_{10}\Phi_{18}.$$

\item If $(\ell_1\ell_3,x)=1,$ then $\ell_4\ell_5\mid x$ or $x$ is
the degree of unipotent characters labeled by 
$(E_6[\theta^i],1)$ or $(E_6[\theta^i],\epsilon),$ where $i=1,2,$
with degrees
$$\frac{1}{3}q^7\Phi_1^6\Phi_2^6\Phi_4^2\Phi_5\Phi_7\Phi_8\Phi_{10}\Phi_{14}\text{ or
}\frac{1}{3}q^{16}\Phi_1^6\Phi_2^6\Phi_4^2\Phi_5\Phi_7\Phi_8\Phi_{10}\Phi_{14}.$$

\item We have $(p\ell_2\ell_5,x)>1,$ and if
$(\ell_4\ell_5,x)=1,$ then $x$ is the degree of  the
unipotent characters labeled by $\phi_{512,11}$ or
$\phi_{512,12}$ with degree
$\frac{1}{2}q^{11}\Phi_2^7\Phi_4^2\Phi_6^3\Phi_8\Phi_{10}\Phi_{12}\Phi_{14}\Phi_{18}.$

\item $H$ possesses two semisimple irreducible characters
$\chi_i,i=1,2$ corresponding to the maximal tori with orders
$\Phi_{18}\Phi_2$ and $\Phi_{14}\Phi_2,$ respectively, and having degrees
$$\Phi_1^7\Phi_2^6\Phi_3^3\Phi_4^2\Phi_5\Phi_6^3\Phi_7\Phi_8\Phi_9\Phi_{10}\Phi_{12}\Phi_{14} \text{ and } 
\Phi_1^7\Phi_2^6\Phi_3^3\Phi_4^2\Phi_5\Phi_6^3\Phi_7\Phi_8\Phi_9\Phi_{10}\Phi_{12}\Phi_{18},$$
respectively such that no proper multiples of these degrees is in
$\cd(L).$

\item The $p$-part of $x$ is at most $q^{46}.$

\item The unipotent character  labeled by the symbol $\phi_{7,1}$
with degree $q\Phi_{7}\Phi_{12}\Phi_{14}$
 is the
smallest nontrivial character degree of $H.$

\item If $x,y\in \cd(L)-\{1,q^{63}\},$ then $(x,y)>1.$

\item $L$ has no consecutive degrees.

\item If $q$ is odd and $q\equiv \epsilon$ \emph{(mod $4$)}, where $q\in\{\pm 1\},$ then
$L$ has an irreducible character $\chi$ of degree
$q^{14}\Phi_1^3\Phi_2^2\Phi_3^2\Phi_5\Phi_6^2\Phi_7\Phi_9\Phi_{10}\Phi_{12}\Phi_{14}\Phi_{18}$
with multiplicity ${(q-\epsilon)}/{4}$ such that $\chi(1)$ divides
no degree of $H.$
Moreover $E_7(q)_{ad}$ has a nontrivial degree $\psi(1)$ which is
not a degree of $H.$ In particular, if $q\equiv 1$ \emph{(mod 4)},
then we can choose
$$\psi(1)=q^{28}\Phi_2^3\Phi_3\Phi_6^2\Phi_9\Phi_{10}\Phi_{12}\Phi_{14}\Phi_{18}$$
and if $q\equiv -1$ \emph{(mod 4)}, then
$\psi(1)=q^{28}\Phi_1^3\Phi_3^2\Phi_{5}\Phi_6\Phi_{7}\Phi_9\Phi_{12}\Phi_{18}.$

\end{enumerate}
\end{lemma}

\begin{lemma}\label{E8} Let $H$ be the simple exceptional group ${\rm E}_8(q)$ and let $\ell_i,i=1,2,\cdots,5,$
be primitive prime divisors of $\Phi_{30},\Phi_{24},\Phi_{20},\Phi_{15}$ and $\Phi_{14},$
respectively. Let $x$ be a
nontrivial character degree of $H$ with $x\neq q^{120}.$
Then the following hold.

\begin{enumerate}[$(i)$]

\item If $(\ell_1\ell_2,x)=1,$ then $x$ is the degree of unipotent
characters labeled by the symbol $E_8[-\theta]$ or $E_8[-\theta^2]$
with degree
$\frac{1}{6}q^{16}\Phi_1^8\Phi_2^6\Phi_3^2\Phi_4^4\Phi_5^2\Phi_7\Phi_8^2\Phi_9\Phi_{10}^2\Phi_{12}\Phi_{14}\Phi_{15}\Phi_{20}.$

\item If $(\ell_2\ell_3,x)=1,$ then $x$ is the degree of  unipotent
characters labeled by the symbol $E_8[i]$ or $E_8[-i]$ with degree
$\frac{1}{4}q^{16}\Phi_1^8\Phi_2^8\Phi_3^4\Phi_5^2\Phi_6^4\Phi_7\Phi_9\Phi_{10}^2\Phi_{14}\Phi_{15}\Phi_{18}\Phi_{30}.$

\item If $(\ell_1\ell_3,x)=1,$ then $x$ is the degree of unipotent
characters labeled by the symbol $E_8[\zeta^k],k=1,2,\cdots,4$ with
degree
$\frac{1}{5}q^{16}\Phi_1^8\Phi_2^8\Phi_3^4\Phi_4^4\Phi_6^4\Phi_7\Phi_8^2\Phi_9\Phi_{12}^2\Phi_{14}\Phi_{18}\Phi_{24}$
or $x$ is the degree of unipotent characters labeled by the symbol
$(D_4,\phi_{1,0})$ or $(D_4,\phi_{1,24})$ with degrees
$\frac{1}{2}q^{3}\Phi_1^4\Phi_3^2\Phi_5^2\Phi_7\Phi_8\Phi_9\Phi_{14}\Phi_{15}\Phi_{24}
,~\frac{1}{2}q^{63}\Phi_1^4\Phi_3^2\Phi_5^2\Phi_7\Phi_8\Phi_9\Phi_{14}\Phi_{15}\Phi_{24},$
respectively.

\item We have $(p\ell_4\ell_5,x)>1,$ and $(p\ell_2\ell_5,x)>1.$

\item $H$ has irreducible characters $\chi_i,i=1,2,3$ with degrees
$|H|/\Phi_{m_i},$ where $m_i$ is $30,24$ and $20,$ respectively.
Moreover no proper multiples of $\chi_i(1),i=1,2,3$  is a degree of
$H.$

\item The $p$-part of $x$ is at most $q^{91}.$

\item The unipotent character with degree
$q\Phi_4^2\Phi_8\Phi_{12}\Phi_{20}\Phi_{24}$ labeled by the symbol
$\phi_{8,1}$
 is the
smallest nontrivial character degree of $H.$

\item If $x,y\in \cd(H)-\{1,q^{120}\},$ then $(x,y)>1.$

\item $H$ has no consecutive degrees.

\end{enumerate}
\end{lemma}


\section{Verifying Step 1: $G'=G''$}\label{sec4}

We verify Step $1$ in this section.   Let $\chi$ be an irreducible character of a finite group $G$. Following \cite{Tong},   $\chi$ is \emph{isolated}
in $G$ if $\chi(1)$ is divisible by no proper nontrivial character
degree of $G,$ and no proper multiple of $\chi(1)$ is a character
degree of $G.$ In this situation, $\chi(1)$ is an
\emph{isolated degree} of $G.$ 

\begin{prop}\label{prop:step1}
Let $G$ be a finite group and let $H$ be one of the finite simple groups in \eqref{eqn1}. Assume that $\cd(G)=\cd(H)$. Then $G'=G''.$
\end{prop}

\begin{proof}
By way of contradiction, suppose
that $G'\neq G''.$ Then there exists a normal subgroup $N$
of $G$ such that $G/N$ is solvable  and minimal with respect to being
nonabelian.  By \cite[Lemma~2.3]{Tong}, $G/N$ is an $r$-group for some
prime $r$ or $G/N$ is a Frobenius group.

\smallskip
{\bf Case $1.$} $G/N$ is an $r$-group. Since
$G/N$ is a nonabelian $r$-group, it has an irreducible
character $\tau$ of degree $r^b$ for some integer $b\ge 1$. By Lemma \ref{lem:pp}, the only nontrivial prime power degree of $H$ is the
degree of the Steinberg character of $H,$ which is $|H|_p.$ It follows that $\tau(1)=r^b=|H|_p$ and hence
$r=p.$ By Thompson's Theorem (\cite[Corollary~12.2]{Isaacs}), $H$
has a nontrivial irreducible character $\psi$ with $p\nmid \psi(1)$.
By \cite[Lemma $2$]{Hupp}, $\psi_N\in\Irr(N)$ and
hence by Gallagher's Theorem (\cite[Corollary~6.17]{Isaacs}), $\tau\psi$ is an irreducible character
of $G$ of degree $\tau(1)\psi(1)=|H|_p\psi(1).$ But this
contradicts \cite[Lemma~2.4]{Tong}.

\smallskip
{\bf Case $2.$} $G/N$ is a Frobenius group with kernel
$F/N,$ $|F/N|=r^a,$ $1<m=|G:F|\in \cd(G)$ and $r^a\equiv 1
(\mbox{mod $m$}).$ By \cite[Lemma~2.4]{Tong},  $\St_H(1)=|H|_p$ is an
isolated degree of $H$ so by \cite[Lemma~2.3(b)(2)]{Tong}, either
$m=|H|_p$ or $r=p.$ We consider each family of simple groups separately.

\smallskip
{\bf Subcase $H\cong {\rm F}_4(q).$} Recall that $q>2$ in this case.
Assume first that $m=|H|_p.$ As $r\nmid m,$  $r\neq p.$
By \cite[$13.9$]{Car85}, $H$ possesses a unipotent character
$\varphi$ labeled by the symbol $\phi_{2,4}'$ with degree
$\varphi(1)=\frac{1}{2}q\Phi_4\Phi_8\Phi_{12}.$ As no proper
multiple of $m$ is a character of $G$ and $\varphi(1)\nmid m,$ we
deduce from \cite[Lemma~2.3(b)(1)]{Tong}  that $r^a\mid \varphi(1)^2.$ As
$r\neq p,$ $r^a\mid \varphi(1)_{p'}^2=
\frac{1}{4}\Phi_4^2\Phi_8^2\Phi_{12}^2.$ Now
$$\frac{1}{2}\Phi_4\Phi_8\Phi_{12}=\frac{1}{2}(q^4+1)(q^6+1)<q^{10}.$$
Hence $r^a\leq \varphi(1)_{p'}^2<q^{20}.$ Since $m\mid r^a-1,$
  $m=q^{24}\leq r^a-1<q^{20}-1<q^{20},$ a contradiction.
  
Thus $r=p.$ Hence $m\in \cd(G)$ is coprime to $p$ and so
$m\neq |H|_p.$ Let $\chi_i\in\Irr(H),i=1,2,$ be irreducible
characters of $H$ obtained from Lemma \ref{F4}$(v)$. For each $i,$ as
no proper multiple of $\chi_i(1)$ is a degree of $H,$ \cite[Lemma~2.3(b)(2)]{Tong} implies that $m$ divides $\chi_i(1)$ and hence
$m$  divides  $\gcd(\chi_1(1),\chi_2(1)).$ Thus $m$ is coprime to $\ell_1\ell_2.$ By Lemma
\ref{F4}$(i)$,
$m=\frac{1}{4}q^4\Phi_1^4\Phi_2^4\Phi_3^2\Phi_6^2,$ which is
impossible as $m$ is coprime to $p.$ 

\smallskip
{\bf Subcase $H\cong {}^2{\rm E}_6(q).$}
Assume first that $m=|H|_p.$ As $r\nmid m,$ $r\neq p.$ By Lemma
\ref{2E6}$(vii)$, $H$ possesses a unipotent character $\varphi$
labeled by the symbol $\phi_{2,4}'$ with degree
$\varphi(1)=q\Phi_8\Phi_{18}.$ As no proper multiple of $m$ is a
character of $G$ and $\varphi(1)\nmid m,$ by \cite[Lemma~2.3(b)(1)]{Tong}, $r^a\mid \varphi(1)^2.$ As $r\neq p,$ $r^a\mid \varphi(1)_{p'}^2= \Phi_8^2\Phi_{18}^2.$ Now $$\Phi_8\Phi_{18}=q^{10}-q^{7}+q^{6}+q^4-q^3+1<q^{15}.$$ As
$r^a\mid \varphi(1)_{p'}^2,$ $r^a\leq
\varphi(1)_{p'}^2<q^{30}.$ Since $m\mid r^a-1,$ 
$$m=q^{36}\leq r^a-1<q^{30}-1<q^{30},$$ which is impossible.

Thus $r=p.$ Then $m\in \cd(G)-\{1,|H|_p\}$ is coprime to
$p.$ By \cite[Lemma~2.3(b)(2)]{Tong}  and Lemma \ref{2E6}$(v),$ $m$
divides both $\chi_i(1),i=1,2.$ As $\chi_i(1)$ is coprime to $\ell_i$
for $i=1,2,$   $m$ is coprime to $\ell_1\ell_2$. By Lemma \ref{2E6}$(i)$, $m=\frac{1}{3}q^{7}\Phi_1^4\Phi_2^6\Phi_4^2\Phi_8\Phi_{10}$, which is impossible
as $p\nmid m.$ 

The arguments for the remaining cases are similar.
\end{proof}

\section{Verifying Step 2: $G'/M$ is isomorphic to $H$}\label{sec5}

Let $H$ be one of the simple groups listed in \eqref{eqn1} and let $G$ be a finite group such that $\cd(G)=\cd(H)$. By Proposition \ref{prop:step1}, $G'=G''$. Let $M\leq G'$ be a normal subgroup of $G$ such that $G'/M$ is a chief factor of $G$. The main purpose of this section is to show that $G'/M\cong H.$  Recall that $H$ is a finite simple exceptional group of Lie type defined over a field of size $q=p^f$, where $p$ is a prime.

\begin{prop}\label{prop:step2}
Let $G$ be a finite group and let $H$ be one of the finite simple groups in \eqref{eqn1}.  Let $M\leq G'$ be a normal subgroup of $G$ such that $G'/M$ is a chief factor of $G$. Then $G'/M\cong H.$ 
\end{prop}

\begin{proof} By Proposition \ref{prop:step1}, $G'$ is perfect. Hence $G'/M\cong S^k$, where $S$ is a nonabelian simple group and $k\ge 1$ is an integer. Using the classification of finite simple groups, we first show that $S$ must be a finite simple groups of Lie type in characteristic $r$ for some prime $r$ and then show that $r=p, k=1$ and $S\cong H$.

\smallskip

(1) Assume  $S=\Alt_n,n\geq 7.$ Let $\theta_i,i=1,2,$ be irreducible
characters of $S$ obtained from Lemma \ref{lem5}$(i)$ with
$\theta_1(1)=n(n-3)/2$ and $\theta_2(1)=(n-1)(n-2)/2.$ Now both $\theta_i$ extend to ${\Aut}(\Alt_n)\cong \Sym_n$ and
$\theta_2(1)=\theta_1(1)+1.$ By Lemma \ref{lem:ext}, 
$\theta_i^k\in \Irr(G'/M)$ extends to $G/M,$ for $i=1,2.$
Hence $\theta_i(1)^k\in \cd(G)$ and $\theta_i(1)^k$ are
coprime for $i=1,2.$ Note that $\Sym_n$ also has an irreducible character labelled by the partition $(n-1,1)$ which is irreducible upon restriction to $\Alt_n$. Thus $\Alt_n$ has an irreducible character $\theta$ of degree $n-1$ which is extendible to $\Aut(\Alt_n)\cong \Sym_n$. Hence $\theta(1)^k=(n-1)^k\in \cd(G)$.

We need the following observation:  $$\Phi_3\Phi_6\Phi_{12}=q^8+q^4+1<
\frac{1}{3}q^4\Phi_1^4\Phi_2^4\Phi_4^2\Phi_{8}=\frac{1}{3}q^4(q^2-1)^2(q^4-1)(q^8-1).$$
If $q=3^f,$ where $f\geq 1,$ then
$\frac{1}{3}q^4\Phi_1^4\Phi_2^4\Phi_4^2\Phi_{8}$ is divisible by
$q^8-1=3^{8f}-1,$ which is  in turn divisible by $3^8-1=2^5\cdot
5\cdot 41,$ and so $\frac{1}{3}q^4\Phi_1^4\Phi_2^4\Phi_4^2\Phi_{8}$
is always divisible by $2^5\cdot 3\cdot 5\cdot 41.$

\smallskip

{\bf Subcase}  $H\cong \textrm{F}_4(q),q>2.$
Assume first that $p\neq 3.$ By Lemma \ref{F4}$(viii),$ one of the
degrees $\theta_i(1)^k,i=1,2,$ must be $q^{24}.$ However, 
$(n-1,n-2)=1$ and $(n,n-3)=(n,3)$ so $\theta_i(1)^k$ can
never be a power of $p\neq 3.$ Thus $p=3.$ If $\theta_j(1)^k=q^{24}$
for some $j,$ then $j=1$ and hence
$\theta_1(1)^k=(n(n-3)/2)^k=3^{24f},$ where $q=3^f.$ As $n\geq 7,$
$n(n-3)/2$ is a prime power if and only if $n=9,$ so
that  $\theta_1(1)^k=3^{3k}=3^{24f}.$ Now $G$ has a degree $(n-1)^k=8^k$. By
Lemma \ref{lem:pp}, the only nontrivial prime power degree
of $G$ is $|H|_p=q^{24}=3^{24f},$ whence $3^{24f}=8^k$, a contradiction. Lemma
\ref{F4}$(viii)$ yields
$$\{\theta_1^k(1),\theta_2^k(1)\}=\{\Phi_3\Phi_6\Phi_{12},\frac{1}{3}q^4\Phi_1^4\Phi_2^4\Phi_4^2\Phi_{8}\}.$$
By Lemma \ref{F4}$(ix)$, $G$ has no consecutive degrees, so $k\geq
2.$ But then $\Phi_3\Phi_6\Phi_{12}$ is a proper nontrivial power,
which is impossible by Lemma \ref{lem7}.

\smallskip
{\bf Subcase} $H\cong {}^2\textrm{E}_6(q).$
Assume first that $p\neq 3.$ By Lemma \ref{2E6}$(viii),$ one of the
degrees $\theta_i(1)^k,i=1,2,$ must be $q^{36}.$ However $(n-1,n-2)=1$ and $(n,n-3)=(n,3)$ so $\theta_i(1)^k$ can
never be a power of $p\neq 3.$ Thus $p=3.$ If $\theta_j(1)^k=q^{36}$
for some $j,$ then $j=1$ and hence
$\theta_1(1)^k=(n(n-3)/2)^k=3^{36f},$ where $q=3^f.$ As $n\geq 7,$
$n(n-3)/2$ is a prime power if and only if $n=9,$ so
that  $\theta_1(1)^k=3^{3k}=3^{36f}.$ Again, $G$ has a degree $(n-1)^k=8^k$. However, by
Lemma \ref{lem:pp}, the only nontrivial prime power degree
of $G$ is $q^{36}=3^{36f},$ so we obtain a contradiction. Thus by Lemma
\ref{2E6}$(viii)$,
$$\{\theta_1(1)^k,\theta_2(1)^k\}=\{\Phi_3\Phi_6^2\Phi_{12}\Phi_{18},
\frac{1}{3}q^7\Phi_1^4\Phi_2^6\Phi_4^2\Phi_{8}\Phi_{10}\}.$$
By Lemma \ref{2E6}$(ix)$, $G$ has no consecutive degrees so
$k\geq 2.$ But then $\Phi_3\Phi_6^2\Phi_{12}\Phi_{18}$ is a proper
nontrivial power. As $q=3^f,$  $\Phi_3$ is coprime to
$\Phi_6^2\Phi_{12}\Phi_{18}$ and so $\Phi_3$ is a nontrivial
proper power, which is impossible by Lemma \ref{lem7}.

The arguments for the remaining groups $ \textrm{E}_6(q),  \textrm{E}_7(q)$ and $ \textrm{E}_8(q)$ are similar. 

\smallskip
(2) Suppose that $S$ is a sporadic simple group or the Tits group. By
Table \ref{Tab2}, $S$ possesses two nontrivial irreducible
characters $\theta_i,i=1,2,$ such that $\theta_1(1)$ and $\theta_2(1)$ are coprime and
each $\theta_i$ extends to $\Aut(S).$ By Lemma \ref{lem:ext}, for each
$i,$ $\theta_i^k$ extends to $G,$ so $\theta_i(1)^k\in\cd(G).$
Observe that for each $i,$ $\theta_i(1)$ is not a prime power so
$\theta_i(1)^k\neq |H|_p.$

Assume $H\cong {\rm F}_4(q)$. As $(\theta_1(1)^k,\theta_2(1)^k)=1,$
by Lemma \ref{F4}$(viii)$, $q=3^f$ and
$$\{\theta_1(1)^k,\theta_2(1)^k\}=\{\Phi_3\Phi_6\Phi_{12},\frac{1}{3}q^4\Phi_1^4\Phi_2^4\Phi_4^2\Phi_{8}\}.$$
In this situation, the larger of $\theta_i(1)$ and $\theta_2(1)$ must
be divisible by $2^5\cdot 3\cdot 5\cdot 41.$ However this is not the
case by inspecting the degrees in Table \ref{Tab2}.

Assume $H\cong {}^2{\rm E}_6(q)$. As $(\theta_1(1)^k,\theta_2(1)^k)=1,$
Lemma \ref{2E6}$(viii)$ yields that $q=3^f$ and
$$\{\theta_1(1)^k,\theta_2(1)^k\}=\{\Phi_3\Phi_6^2\Phi_{12}\Phi_{18},
\frac{1}{3}q^7\Phi_1^4\Phi_2^6\Phi_4^2\Phi_{8}\Phi_{10}\}.$$ In this
situation, one of the degrees $\theta_i(1),i=1,2,$ must be divisible
by $2^5\cdot 3\cdot 5\cdot 41.$ However this is not the case by
inspecting the degrees in Table \ref{Tab2}.

Assume $H\cong {\rm E}_6(q)$. As $(\theta_1(1)^k,\theta_2(1)^k)=1,$
 Lemma \ref{E6}$(viii)$ implies that $q=3^f$ and
$$\{\theta_1(1)^k,\theta_2(1)^k\}=\{\Phi_3^2\Phi_6\Phi_{9}\Phi_{12},
\frac{1}{3}q^7\Phi_1^6\Phi_2^4\Phi_4^2\Phi_{5}\Phi_{8}\}.$$ In this
situation, one of the degrees $\theta_i(1),i=1,2,$ must be divisible
by $2^5\cdot 3\cdot 5\cdot 41.$ However this is not the case by
inspecting the degrees in Table \ref{Tab2}.

Assume $H\cong {\rm E}_7(q)$ or ${\rm E}_8(q)$. Observe that both
$\theta_i(1)$ are not prime powers so both $\theta_1(1)^k$ and $\theta_2(1)^k$ are
pairwise  coprime and different from  $|H|_p,$ which contradicts Lemma \ref{E7}$(viii)$ and Lemma \ref{E8}$(viii)$, respectively.

\smallskip

(3) Assume that $S$ is a finite simple group of Lie type in characteristic $r,$
with $S\neq {}^2{\rm F}_4(2)'.$  Let $\theta$ be the Steinberg
character of $S.$ Then $\theta(1)=|S|_r$ and $\theta$ extends to
${\Aut}(S).$ By Lemma \ref{lem:ext}, $\theta^k\in
\Irr(G'/M)$ extends to $G/M,$ hence $\theta(1)^k=|S|_r^k\in
\cd(G).$ Since $|H|_p$ is the unique nontrivial prime power
character degree of $G$ by Lemma \ref{lem:pp},  $\theta(1)^k=|H|_p.$ In particular,  $r=p.$ 

We next claim that $k=1$. By way of
contradiction, assume $k\geq 2.$ Write
$\theta(1)=q_1^s$ and $|H|_p=q^{a(H)}$, so $a(H)=24,36,36,63,120$, respectively (see \cite[Table 5.1.B]{KL}). Then $q_1^{sk}=q^{a(H)}.$ Let $\psi=\tau\times
\theta\times\cdots\times \theta\in \Irr(S^k),$ where
$\tau\in \Irr(S)$ with $1<\tau(1)\neq \theta(1).$ Then
$\psi(1)=\tau(1)q_1^{s(k-1)}\nmid q^{a(H)}$ and is nontrivial, so
it divides some character degree of $G,$ which is different from
$q^{a(H)}.$ By part $(vi)$ of Lemmas \ref{2E6}--\ref{E8}, 
$q_1^{s(k-1)}=q^{a(H)(k-1)/k}\leq q^{b(H)}$, where $b(H)=25,16,25,46,91$, respectively, and hence $a(H)(k-1)\leq kb(H),$
which implies that $k\leq k_0$ for some positive integer $k_0$. We  check that  $k\leq k_0$, where $k_0= 3$ unless $H={\rm E}_8(q)$ where $k_0=4$.

Let $C$ be a normal subgroup of $G$ such that $C/M=C_{G/M}(G'/M).$
Then $G'C/C\cong S^k$ is a unique minimal normal subgroup of $G/C$,
so $G/C$ embeds into ${\Aut}(S)\wr S_k,$ where $S_k$ is the
symmetric group of degree $k.$ Let $B={\Aut}(S)^k\cap G/C.$ Then
$|G/C:B|\mid k!.$ Let $\psi=1\times \theta\times\cdots\times
\theta\in \Irr(G'C/C).$ Then $\psi$ extends to
$\psi_0\in\Irr(B)$. Let $J=I_G(\psi_0)$ and let $\mu\in\Irr(J|\psi_0).$ By
\cite[Theorem~6.11]{Isaacs}, $\chi=\mu^G\in \Irr(G).$ Then
$$\chi(1)=e\psi_0(1)=e\psi(1)=eq^{a(H)(k-1)/k},$$ where $e\mid k!.$
Since $k\leq 4,$ we deduce that $e\mid 4!=24,$ so
$1<\chi(1)<q^{a(H)}.$ In addition $\chi(1)$ is divisible by $q^{\lfloor a(H)/2\rfloor}.$
As $\ell_1$ and $\ell_2$ are primitive prime divisors of $\Phi_{m_1}$
and $\Phi_{m_2},$ where $m_1>m_2$ are the two largest integers such that $\Phi_{m_i}$ divides $|H|$, it follows from Lemma \ref{Zsigmondy}
that $\ell_1\geq m_1+1$ and $\ell_2\geq m_2+1$. In particular, $\ell_1,\ell_2\ge 11$. So
$(\ell_1\ell_2,\chi(1))=1.$ By part $(i)$ of Lemmas \ref{2E6}--\ref{E8},  $\chi(1)=\frac{1}{t}q^{c(H)}\alpha,$ where $t=3,4,3,2,6$ and $c(H)=7,4,7,11,16$, respectively and $\alpha$ is an integer coprime to $p$. However, this is a contradiction since  $\chi(1)$ is divisible by $q^{\lfloor a(H)/2\rfloor}>q^{c(H)}$. Thus  $k=1$ and $r=p.$

\smallskip

(4) {Eliminating finite simple groups of Lie type in characteristic $p$.} We show that if $G'/M\cong S$ is a simple group of Lie type in
characteristic $p$ and $S\neq {}^2{\rm F}_4(2)',$ then $S\cong H.$ We
prove this by eliminating other possibilities of $S.$ Assume
that $S$ is a simple group of Lie type in characteristic $p$ and $S$
is not the Tits group. We have shown that $G'/M\cong S$ and
$|S|_p=|H|_p=q^{a(H)}.$ Observe that if $\theta\in \Irr(S)$ is
extendible to ${\Aut}(S),$ then $\theta$ extends to $G/C,$ where
$C/M=C_{G/M}(G'/M),$ so $\theta(1)\in \cd(G/C)\subseteq
\cd(G).$ In fact, we choose $\theta$ to be a unipotent
character of $S.$ By  \cite[Theorem $2.5$]{Malle08}, $\theta$ is
extendible to ${\Aut}(S)$ apart from some exceptions. We refer to
\cite[$13.8,13.9$]{Car85} for the classification of unipotent
characters and the notion of symbols. In Table \ref{Tab1}, for each
simple group of Lie type $S$ in characteristic $p,$ we list the
$p$-part of  some unipotent character  of $S$ that is extendible to
${\Aut}(S).$

\smallskip

{\bf Subcase }$H\cong \textrm{F}_4(q),q>2.$

$(a)$ Case $S\cong \LL_n^\epsilon(p^b),$ with $n\geq
2.$ Here $bn(n-1)=48f.$ If $n=2$ then $b=24f$, so
$S=\LL_2(q^{24})$. Hence $S$ has a character of degree $q^{24}+1.$
Obviously this degree divides no degree of $G$ since $q^{24}+1\nmid
|{\rm F}_4(q)|.$ If $n=3$ then $b=8f$ and so $S=\LL_3^\epsilon(q^8).$
By \cite[$(13.8)$]{Car85}, $S$ possesses a unipotent character
parametrized by the partition $(1,2)$ of degree $q^8(q^8+\epsilon
1).$ Since $q^8+1$ does not divide the order of $H$, the case
when $\epsilon=+$ cannot occur. Now assume $\epsilon=-.$ Then
$q^8(q^8+\epsilon 1)=q^8(q^8-1)=q^8\Phi_1\Phi_2\Phi_4\Phi_8,$ which
is coprime to $\ell_3\ell_4,$ so  $q^8(q^8-1)$ must be
one of the degrees given in case $(iv)$ of Lemma \ref{F4}. However
by comparing the power of $q,$ we see that this case is impossible.
If $n=4,$ then $b=4f$ so $S=\LL_4^\epsilon(q^4).$ In this case,
the unipotent character parametrized by the partition $(2,2)$ has
degree $q^8(q^8+1).$ As above, this degree does not belong to
$\cd(G).$ Thus we  assume that $n\geq 5.$ By Table
\ref{Tab1}, $S$ possesses a unipotent character $\chi\neq \St_S$  with $\chi(1)_p=p^{b(n-1)(n-2)/2}.$ By
Lemma \ref{F4}$(vi)$,  if $p$ is odd, then  $b(n-1)(n-2)/2\leq 16f$; else
$b(n-1)(n-2)/2\leq 13f$. Assume first that $p$ is
even. Multiplying both sides of the inequality $b(n-1)(n-2)/2\leq
13f$ by $n,$ we obtain $$bn(n-1)(n-2)/2=24f(n-2)\leq 13nf,$$ and so
$24(n-2)\leq 13n.$ Thus $11n\leq 48$ so $ n\leq 4,$ a
contradiction. Thus $p$ is odd. Using the same argument with
$b(n-1)(n-2)/2\leq 16f,$ we obtain that $n\leq 6.$ Hence $5\leq
n\leq 6$ and $p$ is odd. Assume that $n=5.$ Then $5b=12f.$ Since the primitive prime divisor $\ell_{4b}(p)$ exists,
$\ell_{4b}(p)\mid |S|,$ and hence $\ell_{4b}(p)\mid |H|.$ Since
$4b=48f/5>8f$, $4b\mid 12f=5b,$ which is
impossible. Assume $n=6.$ Then $5b=8f$ and the primitive prime
divisor $\ell_{6b}(p)\in\pi(S)\subseteq \pi(H)$ exists. As
$6b>5b=8f$,  $6b\mid 12f,$ and hence $b\mid 2f$. It follows that $5b=8f\mid 10f,$ which is impossible.

$(b)$ Case $S\cong \SSS_{2n}(q_1),$ or $\OO_{2n+1}(q_1),$ where
$q_1=p^b,$  $n\geq 2$ and $S\neq \SSS_4(2).$ Here
$bn^2=24f.$ If $n=2$ then $b=6f$ and so $q_1=q^6.$ By
\cite[$(13.8)$]{Car85}, $S$ possesses a unipotent character labeled
by the symbol $\binom{0\:1\:2}{\:\:-\:\:}$ of degree
$\frac{1}{2}q^6(q^{6}-1)^2.$ Now
$\frac{1}{2}q^6(q^6-1)^2\neq q^{24}$ and it is coprime to both
$\ell_1$ and $\ell_2$. By Lemma \ref{F4}(i), 
$\frac{1}{2}q^6(q^6-1)^2=\frac{1}{4}q^4\Phi_1^4\Phi_2^4\Phi_3^2\Phi_6^2.$
This equation is impossible by comparing the power of $p.$

If $n=3$ then $3b=8f$ and so $q_1^{3}=q^8.$ By
\cite[$(13.8)$]{Car85}, $S$ possesses a unipotent character $\chi$
labeled by the symbol $\binom{1\:2}{\:1\:}$ with degree
$${q_1^3(q_1^6-1)}/({q_1^2-1})={q^8(q^{16}-1)}/({q^{16/3}-1}).$$ However $\chi(1)$ is not a degree of $G,$
since $\ell_{16}(q)\mid \chi(1)$ but  $\ell_{16}(q)\nmid |H|.$

If $n=4$ then $2b=3f$ so $b\geq 3.$ Hence
$\ell_{6b}(p)\in\pi(S)\subseteq \pi(H),$ and $6b=9f>8f,$ and we deduce
that $6b=9f\mid 12f,$ which is impossible.

If $n=5$ then $25b=24f.$ Now $10b=240f/25=48f/5>9f$ so
$\ell_{10b}(p)\in\pi(S)\subseteq \pi(H),$ hence $10b\mid 12f,$ which
implies that $5b\mid 6f$ and so $25b=24f\mid 30f,$ which is
impossible.

Assume $n\geq 6.$ By Table \ref{Tab1}, $S$
possesses a nontrivial irreducible character $\chi\neq \St_S$  with $\chi(1)_p=p^{b(n-1)^2}/(2,p).$ Assume
that $p=2.$ Then $\chi(1)_p=p^{b(n-1)^2-1}$. By Lemma
\ref{F4}$(vi),$  $$bn(n-2)\leq b(n-1)^2-1 \leq 13f.$$
Multiplying both sides by $n$ and simplifying,  $11n\leq 48$
so $n<5,$ a contradiction. Now assume $p>2.$ Then
$\chi(1)_p=p^{b(n-1)^2}$. By Lemma \ref{F4}$(vi),$  $b(n-1)^2\leq 16f.$ Multiplying both sides by $3,$ 
$$3b(n-1)^2\leq 48f=2bn^2.$$ After simplifying,  $n(n-6)+3\leq
0,$ which is a contradiction as $n\geq 6.$

$(c)$ Case $S\cong \OO_{2n}^\epsilon(q_1),$ where $q_1=p^b,$ and $n\geq 4.$ Here $bn(n-1)=24.$ If $n=4,$ then $b=2f$ hence
$q_1=q^2.$ If $S=\OO_8^+(q^2),$ then $S$ has a unipotent character
$\chi$  labeled by the symbol in Table \ref{Tab1} with
$\chi(1)=q^{14}\Phi_8^2$. If $S=\OO_8^-(q^2),$ then $S$ has a
unipotent character $\chi$ labeled by the symbol $\binom{1~3}{~-}$
with degree $\chi(1)={(q^8+1)(q^6-q^2)}/{(q^4-1)}.$ However, these
degrees do not divide $|H|.$ If $n=5,$ then $5b=6f.$ Since $8b=48f/5>9f$, 
$\ell_{8b}(p)\in\pi(S)\subseteq\pi(H),$ so $8b\mid 12f=10b,$
which is impossible. If $n=6,$ then $5b=4f.$ Now
$8b=32f/5>6f$ and so $\ell_{8b}(p)\in\pi(S)\subseteq\pi(H),$ thus
$8b\mid 8f=10b$ or $8b\mid 12f=15b.$ Both cases are
impossible. Thus we can assume that $n\geq 7.$ By Table \ref{Tab1},
$S$ possesses a unipotent character $\chi\neq \St_S$  with $\chi(1)_p\geq p^{b(n-1)(n-2)}.$ By Lemma
\ref{F4}$(vi),$  $b(n-1)(n-2)\leq 16f.$ Multiplying both
sides by $n,$ and simplifying,   $8n\leq 48$ and hence
$n\leq 6,$ a contradiction.

$(d)$ Case $S\cong {\rm F}_4(p^b),$ where $b\geq 1.$ Then $24b=24f,$ hence
$b=f$ so $S\cong H.$

$(e)$ Case $S\cong {\rm G}_2(q_1),$ where $q_1=p^b$.  Here
$6b=24f$ and so $b=4f.$ Thus $S={\rm G}_2(q^4),$ so $S$ has an
irreducible character of degree $q^{24}-1$ by \cite{Chang}. This degree divides no degrees of $H.$

$(f)$ Case $S\cong {}^3{\rm D}_4(p^b)$. Then $12b=24f$
and hence $b=2f.$ By \cite[13.9]{Car85}, $S$ possesses a unipotent
character $\chi$ labeled by the symbol $\phi_{1,3}''$ with degree
$\chi(1)=q^{14}\Phi_{24}.$ But then this degree does not divide the
order of $H.$

$(g)$ Case $S\in \{{}^2{\rm B}_2(q_1^2),{}^2{\rm G}_2(q_1^2),{}^2{\rm F}_4(q_1^2)\},$
where $q_1^2=2^{2n+1},3^{2n+1},2^{2n+1}$ with $n\geq 1,$
respectively. For each case, the equations $|S|_p=p^{24f}$ is
impossible.

$(h)$ For the remaining cases, we argue as follows. In all cases
$|S|_p=p^{24f}.$ Lemma \ref{F4}$(vi)$ yields that $\chi(1)_p\leq
p^{16f},$ where $\chi$ is the unipotent character listed in Table
\ref{Tab1}. Using these two properties, we obtain a
contradiction. For instance, assume that $S\cong {\rm E}_8(p^b),$ where
$b\geq 1.$ Then $120b=24f$ and hence $5b=f.$ By Table \ref{Tab1},
$S$ possesses a unipotent character $\chi\neq \St_S$ with $\chi(1)_p\geq p^{91b}.$ By Lemma
\ref{F4}$(vi),$  $91b\leq 16f=80b,$ a contradiction. 

\smallskip

{\bf Subcase} $H\cong {}^2\textrm{E}_6(q).$

$(a)$ Case $S\cong \LL_n^\epsilon(p^b),$ with $n\geq
2.$ Here $bn(n-1)=72f.$ If $n=2$ then $b=36f$ so
$S=\LL_2(q^{36})$ and hence $S$ has a character of degree $q^{36}+1.$
However this is impossible as this number does not divide the order
of $H.$ Hence we  assume that $n\geq 3.$ If $(n,q)=(6,2)$ or
$(3,2),$ then $bn(n-1)\neq 72f$.
Thus $\ell_{\epsilon bn}(p)\in\pi(S)\subseteq\pi(H)$ always exists
by Lemma \ref{Zsigmondy}, and so $bn\leq 18f.$ Multiplying both
sides by $n-1$ and simplifying, $n\geq 5.$ By Table
\ref{Tab1}, $S$ possesses a unipotent character $\chi\neq\St_S$ with $\chi(1)_p=p^{b(n-1)(n-2)/2}.$ By
Lemma \ref{2E6}(vi),  $b(n-1)(n-2)/2\leq 25f.$ Multiplying
both sides of the latter inequality by $n$ and simplifying,  $ n\leq 6.$ If $n=5,$ then $5b=18f$ and hence
$4b=72f/5>14f$ so $\ell_{4p}(p)\in\pi(S)\subseteq\pi(H)$ exists
and thus $4b\mid 18f=5b,$ which is impossible. Similarly if $n=6,$
then $5b=12f$ and hence $6b=72f/5>14f$ so
$\ell_{6p}(p)\in\pi(S)\subseteq\pi(H)$ exists and thus $6b\mid 18f.$
Hence $12b\mid 36f=15b,$ a contradiction.

$(b)$ Case $S\cong \SSS_{2n}(q_1),$ or $\OO_{2n+1}(q_1),$ with $q_1=p^b,$
 $n\geq 2$ and $S\neq \SSS_4(2).$ Now $bn^2=36f.$ If
$p=2$ and $2nb=6,$ then $n=3,b=1$ and so $bn^2=9=36f,$ a
contradiction. Thus $\ell_{2bn}(p)\in\pi(S)\subseteq\pi(H)$ exists.
Hence $2bn\leq 18f.$ Multiplying both sides by $n$ and simplifying,
 $n\geq 4.$  By Table \ref{Tab1}, $S$ possesses a
nontrivial irreducible character $\chi\neq\St_S$ with $\chi(1)_p\geq p^{bn(n-2)}$. By Lemma
\ref{2E6}$(vi)$,  $bn(n-2) \leq 25f.$ Multiplying both
sides by $n$ and simplifying, $11n\leq 72$ so $n\leq
6.$ If $n=4,$ then $4b=9f$. Since $6b=54f/4>12f$,
$\ell_{6p}(p)\in\pi(S)\subseteq\pi(H)$ exists and thus $6b\mid
18f=8b,$ which is impossible. Similarly if $n=5,$ then $25b=36f$. Since $10b=360f/25>14f$,
$\ell_{10p}(p)\in\pi(S)\subseteq\pi(H)$ exists and thus $10b\mid
18f.$ It follows that $20b\mid 36f=25b,$ which is impossible. Thus
$n=6$ and so $b=f.$ By \cite{Car85}, $S$ possesses a unipotent
character $\chi^\alpha$ labeled by the symbol
$\alpha=\binom{0~1~6}{~-~}$ with degree
$\chi^\alpha(1)=(q^6-1)(q^6-q)/(2(q+1)).$ Obviously $\chi^\alpha(1)$
is nontrivial and different from $|H|_p$ and furthermore it is coprime
to $\ell_i$ for $i=1,2,3,$ contradicting Lemma \ref{2E6}$(i)$.

$(c)$ Case $S\cong \OO_{2n}^\epsilon(q_1),$ where $q_1=p^b,$ and $n\geq 4.$ Here $bn(n-1)=36f.$ By Table \ref{Tab1}, $S$
possesses a unipotent character $\chi\neq\St_S$ with $\chi(1)_p\geq p^{b(n-1)(n-2)}.$ By Lemma
\ref{2E6}$(vi)$,  $b(n-1)(n-2)\leq 25f.$ Multiplying both sides
by $n,$ and simplifying,  $n\leq 6.$ If $n=4,$ then $b=3f$
and hence $q_1=q^3.$ By \cite{Car85}, $S$ possesses a unipotent
character $\varphi$ labeled by the symbol $\binom{1~n}{0~1}$ when
$\epsilon=+$ and $\binom{0~1~n}{~~1~~}$ when $\epsilon=-$ of degree
$$(q_1^{2n}-q_1^2)(q_1^2-1)=(q_1^8-q_1^2)/(q_1^2-1)=q^6\Phi_9\Phi_{18}.$$
It follows that $\varphi(1)\in\cd(L)-\{1,q^{36}\}$ and
$(\varphi(1),\ell_2\ell_3)=1.$ By Lemma \ref{2E6}$(ii)$, $\varphi(1)$
is one of the degrees in $(ii).$ Both cases are
impossible by comparing the power on $q.$ If $n=5,$ then $5b=9f$ and
hence $8b=72f/5>14f$ so $\ell_{8p}(p)\in\pi(S)\subseteq\pi(H)$
exists and thus $8b\mid 18f=10b,$ which is impossible. Thus $n=6$
and so $5b=6f.$  As when $n=4,$ $S$ possesses a unipotent
character $\varphi$ of degree
$$(q_1^{12}-q_1^2)/(q_1^2-1)=q_1^2(q_1^{10}-1)/(q_1^2-1)=q_1^2(q^{12}-1)/(q_1^2-1).$$
It follows that $\varphi(1)\in\cd(L)-\{1,q^{36}\}$ and
$(\varphi(1),\ell_1\ell_3)=1,$ contradicting Lemma
\ref{2E6}$(iii)$.

$(d)$ Case $S\cong {\rm G}_2(p^b),$ where  $b\geq 1.$ Here $6b=36f.$ By \cite{Chang}, $S$ possesses an
irreducible character of degree $p^{6b}-1=q^{36}-1.$ However $q^{36}-1$ does not divide $|H|,$
 a contradiction.

$(e)$ Case  $S\in\{ {}^2{\rm B}_2(q_1^2),{}^2{\rm G}_2(q_1^2),$ where $q_1^2=2^{2n+1},3^{2n+1},$ respectively with
$n\geq 1.$ In these cases, the equation $|S|_p=|H|_p$ cannot occur.

$(f)$ Case  $S\cong {}^2{\rm F}_4(q_1^2),$ where $q_1^2=2^{2n+1},$ with
$n\geq 1.$ Then $12(2n+1)=36f.$ Now
$\ell_{12(2n+1)}(2)\in\pi(S)\subseteq \pi(H),$ which implies that
$12(2n+1)=36f\leq 18f,$ a contradiction.

$(h)$ Case $S\cong {\rm F}_4(p^b)$. Then $24b=36f$ and so $2b=3f.$ By
\cite[13.9]{Car85}, $S$ possesses a unipotent character $\varphi$ labeled by the symbol
$\phi_{9,2}$ with $\varphi(1)=q_1^2\Phi_3^2(q_1)\Phi_6^2(q_1)\Phi_{12}(q_1)=q^3\Phi_9^2\Phi_{18}.$
Hence $\varphi(1)\in\cd(L)-\{1,q^{36}\}$ and $(\varphi(1),\ell_1\ell_3)=1,$ 
contradicting Lemma \ref{2E6}$(iii)$.

$(i)$ Case $S\cong {}^3{\rm D}_4(p^b)$. Then $12b=36f.$
Now $\ell_{12b}(p)\in\pi(S)$ but
$\ell_{12b}(p)=\ell_{36f}(p)\not\in\pi(H)$, a contradiction.

$(j)$  Case $S\cong {}^2{\rm E}_6(p^b)$. Then $36b=36f$
and hence $b=f$ so $S\cong H.$

$(k)$ Case $S\cong {\rm E}_7(p^b)$. Then $63b=36f$ and hence $7b=4f.$ By Table
\ref{Tab2}, $S$ possesses a unipotent character $\varphi\neq \St_S$  with $\varphi(1)_p=p^{46b}.$ By Lemma \ref{2E6}$(vi)$,  $46b\leq 25f,$ which
implies that $184b=4\cdot 46b\leq 100f=175b,$ which is impossible.

$(l)$  Case $S\cong {\rm E}_8(p^b)$. Then $120b=36f$ and hence $10b=3f.$
By Table \ref{Tab2} and Lemma \ref{2E6}$(vi)$,  $91b\leq 25f.$ It follows that $273b\leq 25\cdot 3f=250b,$ a contradiction.

$(m)$  Case $S\cong {\rm E}_6(p^b)$.  Then $36b=36f$ and
hence $b=f.$ By \cite[13.9]{Car85}, $S$ has a unipotent character
$\varphi$ labeled by the symbol $\phi_{6,1}$ with degree
$q\Phi_8\Phi_9.$ Now $\varphi(1)\in\cd(L)-\{1,|H|_p\}$  is
coprime to $\ell_1\ell_2\ell_4,$  contradicting Lemma
\ref{2E6}$(i)$.


The arguments for the remaining simple groups ${\rm E}_6(q), {\rm E}_7(q)$ and ${\rm E}_8(q)$ are similar, so we skip the details.
\end{proof}

\section{Verifying Step 3: $G'$ is isomorphic to $H$}\label{sec6}

We prove the following result in this section.
\begin{prop}\label{prop:step3}
Let $G$ be a finite group and let $H$ be one of the finite simple groups in \eqref{eqn1}.  Let $M\leq G'$ be a normal subgroup of $G$ such that $G'/M$ is a chief factor of $G$. Then $G'\cong H.$
\end{prop}

\begin{proof}
By Propositions \ref{prop:step1} and  \ref{prop:step2}, $G'$ is perfect and $G'/M\cong H.$ It suffices to show that $M=1.$
By way of contradiction, assume that $M$ is nontrivial. Let $K\leq
M$ be a normal subgroup of $G'$ such that $M/K$ is a chief factor of
$G'.$ Assume first that $M/K$ is nonabelian. Then $M/K\cong S^k,$ where $k\geq 1$ and $S$ is a nonabelian finite simple
group. By Lemmas \ref{lem5} and \ref{lem:ext}, $M/K$ possesses a
nontrivial irreducible character $\varphi\in\Irr(M/K)$ which extends
to $\varphi_0\in\Irr(G').$ As $G'/M\cong H,$ it possesses an
irreducible character $\chi$ of degree $|H|_p,$ and hence by
Gallagher's Theorem,  $\varphi_0\chi$ is an irreducible character of
$G'$ of degree $\varphi_0(1)|H|_p>|H|_p.$ As $G'\unlhd G,$  $\varphi_0(1)|H|_p$ divides some degree of $G,$ 
contradicting \cite[Lemma~2.4]{Tong}. Thus $M/K$ is a 
minimal normal elementary abelian subgroup of $G'/K.$ Let $C/K$ be
the centralizer of $M/K$ in $G'/K.$ As $M/K$ is a minimal normal
elementary abelian subgroup of $G'/K,$  we deduce that $M/K\leq
C/K\unlhd G'/K.$ In particular, $M\unlhd C\unlhd G'.$ As $G'/M$ is
nonabelian simple, either $C=G'$ or $C=M.$

Assume  $C=M.$ Then $M/K$ is a self-centralizing minimal
normal elementary abelian subgroup of $G'/K$ and $G'/M\cong H.$
Let $\chi\in \Irr(G'/K)$ be such that $\chi(1)$ is the largest
degree of $G'/K.$  Applying Lemma \ref{lem8} to $M/K\unlhd G'/K,$ we obtain $\chi(1)>b(H).$ Now
$\chi(1)\in\cd(G')$ divides some degree of $H$ and thus
$\chi(1)\leq b(H)$, a contradiction.

Assume  $C=G'$. Then $M/K=Z(G'/K).$ Since $G'$ is
perfect,  $G'/K$ is perfect and so $M/K=(G'/K)'\cap
Z(G'/K),$ which implies that $G'/K$ is a central extension of $G'/M$
(see \cite[p.~629]{HuppI}).  We now consider each case separately.

\smallskip
{\bf Subcase} $H\cong \textrm{F}_4(q), (q>2)$ or ${\rm E}_8(q)$.
 By \cite[Theorem~5.1.4]{KL}, the Schur multiplier of $G'/M\cong H$ is
trivial, so $G'/K\cong G'/M,$ which implies that $M/K$
is trivial, a contradiction.

\smallskip
{\bf Subcase} $H\cong {}^2\textrm{E}_6(q).$
 By \cite[Theorem~5.1.4]{KL}, the Schur multiplier of $G'/M\cong H$  is cyclic of order
$d=(3,q+1).$ If  $d=1,$ then $G'/K\cong H$ so $M/K$ is trivial,
a contradiction. Hence $d=3.$ Assume first that $q>2.$ Then
${}^2{\rm E}_6(q)_{sc}$ is the universal covering group of $H$ so
$G'/K\cong {}^2{\rm E}_6(q)_{sc}.$ As $G'\unlhd G$ and $\cd(G'/K)\subseteq
\cd(G'),$ every degree of $G'/K$  divides some
degree of $G,$ which is impossible by Lemma \ref{2E6}$(x)$. Now assume
that $q=2.$ The Schur multiplier of ${}^2{\rm E}_6(2)$ is
isomorphic to $\Z_2\times \Z_6.$
If $G'/K$ is isomorphic to $2\cdot {}^2{\rm E}_6(2)$ or $6\cdot {}^2{\rm E}_6(2),$ then $G'/K$ possesses an
irreducible character of degree $2^7\cdot 3^5\cdot 5^2\cdot 7^2\cdot 11\cdot 13\cdot 19$ which
divides no degree of ${}^2{\rm E}_6(2)$. If $G'/K\cong 3\cdot {}^2{\rm E}_6(2),$ then $G'/K$ has an
irreducible character of degree $2^{11}\cdot 3^3\cdot 5\cdot 7^2\cdot 13\cdot 17\cdot 19.$
Thus this case cannot occur.

The remaining cases can be argued similarly, we skip the details.
\end{proof}


\section{Verifying Step 4: $G=G' \times C_G(G')$ }\label{sec7}

In this section, we  complete the proof of the main theorem.
We first recall the setup in Section \ref{sec3}. Let $\mathcal{G}$ be a
simple simply connected algebraic group in characteristic $p$ and $F$ is a Steinberg endomorphism such that
$\mathcal{G}^F\cong H_{sc}=L,$ where $H\cong L/Z(L)$ is one of the simple groups in \eqref{eqn1}. Denote by $\mathcal{G}^*$ the dual group of $\mathcal{G}$ and let
$F^*$ be the corresponding Frobenius automorphism. Then $H^*:=(\mathcal{G}^*)^{F^*}\cong
H_{ad}.$  
\begin{prop}\label{prop:step4}
Let $G$ be a finite group and let $H$ be one of the finite simple groups in \eqref{eqn1}.  Suppose that $\cd(G)=\cd(H).$ Then $G\cong G'\times C_G(G').$
\end{prop}

\begin{proof}
By Proposition \ref{prop:step3},  $G'\cong H$.
Let $C=C_G(G').$ Then $G/C$ is almost simple with socle $H$ and $G'\cap C=1$. If $G=G'C$, then $G\cong G' \times C$ and we are done. So, assume by contradiction that $G'\times C<G.$ Then $G$ induces some outer automorphism, say $\sigma$,
on $H.$

\smallskip
{\bf Subcase} $H\cong \textrm{F}_4(q),q>2.$ In this case, $H_{sc}\cong H\cong H_{ad}.$
When $p$ is odd, $\textrm{Out}(H)$ is a cyclic
group of order $f,$  consisting of field automorphisms, and when $p$
is even, $\textrm{Out}(H)$ is a cyclic group of order $2f$ with
generator $\gamma$ such that $\gamma^2$ is a generator for the group
of field automorphisms of $H.$

Assume that $\sigma$ is a nontrivial field automorphism of
$H.$  Let
$\ell=\ell_{12f}(p)$ be a primitive prime divisor of $p^{12f}-1$ and
let $s$ be a semisimple element of order $\ell$ in
$(\mathcal{G}^*)^{F^*}\cong H.$ Now
$|C_{(\mathcal{G}^*)^{F^*}}(s)|=\Phi_{12}(q)$ and $H$ possesses an
irreducible character $\chi_s$ with degree
$$|(\mathcal{G}^*)^{F^*}:C_{(\mathcal{G}^*)^{F^*}}(s)|_{p'}=\Phi_1^4\Phi_2^4\Phi_3^2\Phi_4^2\Phi_6^2\Phi_8.$$
Let $\sigma^*$ be an automorphism of $\mathcal{G}^*$ induced from
the field automorphism $\sigma.$ Then
$(\mathcal{G}^*)^{\sigma^*}\cong {\rm F}_4(p^b),$ where $1\leq b<f,$ so
$\ell=\ell_{12f}(p)$ does not divide
$|(\mathcal{G}^*)^{\sigma^*}|$. By
\cite[Lemma~2.5(ii)]{Dolfi}, $\chi_s$ is not
$\sigma$-invariant, and hence some proper multiple of $\chi_s(1)$ is
a degree of $H,$ contradicting Lemma \ref{F4}$(v)$.

Now assume that $\sigma$ is not a field automorphism of $H.$ Then
$p=2$ and $\sigma=\gamma^k$ for some odd integer $1\leq k\leq f,$
where $\gamma$ is a generator of $\Out(H).$ Let
$\tau=\sigma^2=\gamma^{2k}.$ If $\tau$ is nontrivial, then it is a
field automorphism of $H$. By applying the same argument as above,
we obtain that $\chi_s$ is not $\tau$-invariant and so obviously
$\chi_s$ is not $\sigma$-invariant, a contradiction.
Hence  $\sigma^2=1,$ which forces $f=k$ is odd. Thus
$\sigma$ is a graph automorphism of order $2$ of ${\rm F}_4(2^f),$ where
$f\geq 2$ is odd. By \cite[Theorem~2.5(e)]{Malle08}, $H$ has two
unipotent characters labeled by the symbol $\phi_{2,16}'$ and
$\phi_{2,16}''$ with degree
$\frac{1}{2}q^{13}\Phi_4\Phi_8\Phi_{12},$ which are fused under
$\sigma.$ Therefore $G/C$ possesses an irreducible character of
degree $q^{13}\Phi_4\Phi_8\Phi_{12}.$ However this contradicts Lemma
\ref{F4}$(vi)$ as $q$ is even.

\smallskip
{\bf Subcase} $H\cong {}^2\textrm{E}_6(q).$
Note that $\textrm{Out}(H)$ is a group of
order $2df,$ where $d=(3,q+1).$
Assume first that $G/C$ possesses only inner and diagonal
automorphism. It follows that $G/C\cong {}^2{\rm E}_6(q)_{ad},$ $3\mid (q+1)$  and
$\sigma$ is a diagonal automorphism of $H.$ Thus
$\cd({}^2{\rm E}_6(q)_{ad})=\cd(G/C)\subseteq\cd(G)=\cd(H),$ which is
impossible by Lemma \ref{2E6}$(x)$.

Hence we can assume that $\sigma=\gamma\varrho,$ where $\varrho$ is
a field automorphism and $\gamma$ is either trivial or a diagonal
automorphism. 
Let $\ell=\ell_{18f}(p)$ be a primitive prime divisor of $p^{18f}-1$
and let $s$ be a semisimple element of order $\ell$ in
$(\mathcal{G}^*)^{F^*}.$ Now
$|C_{(\mathcal{G}^*)^{F^*}}(s)|=\Phi_{18}$ and $\mathcal{G}^F$
possesses an irreducible character $\chi_s$ with degree
$$|(\mathcal{G}^*)^{F^*}:C_{(\mathcal{G}^*)^{F^*}}(s)|_{p'}=|{}^2{\rm E}_6(q)_{ad}|_{p'}/\Phi_{18}$$
which acts trivially on its center, so $\chi_s\in\Irr(H).$  Let
$\varrho^*$ be an automorphism of $\mathcal{G}^*$ induced from the
field automorphism $\varrho.$ Then
$(\mathcal{G}^*)^{\varrho^*}\cong {}^2{\rm E}_6(p^b)_{ad},$ where $1\leq b<f,$
so $\ell=\ell_{18f}(p)$ does not divide
$|(\mathcal{G}^*)^{\varrho^*}|,$ and hence by
\cite[Lemma~2.5(ii)]{Dolfi},  $\chi_s$ is not
$\varrho$-invariant. As $\chi_s$ is invariant under the diagonal
automorphisms,  $\chi_s$ is not $\sigma$-invariant and
so some proper multiple of $\chi_s(1)$ is a degree of $H,$ which
is impossible.

\smallskip
{\bf Subcase} $H\cong \textrm{E}_6(q).$
Note that $\textrm{Out}(H)\cong \la\gamma\ra (\la\varrho\ra\times \la\tau\ra))$ is a group of order
$2df,$ where $d=(3,q-1),$ and $\gamma$ is a diagonal automorphism of order $d,$ and $\varrho$ is a field
automorphism of order $f$ and $\tau$ is a graph automorphism of order $2$ (see for example
\cite[Theorem~2.5.12]{Gorenstein}). Moreover $\tau$ inverts the cyclic group $\la\gamma\ra.$

Assume first that $G/C$ possesses only inner and diagonal
automorphism. It follows that $G/C\cong {\rm E}_6(q)_{ad},$ $3\mid (q-1)$  and
$\sigma$ is a diagonal automorphism of $H.$ Thus 
$\cd({\rm E}_6(q)_{ad})=\cd(G/C)\subseteq\cd(G)=\cd(H),$ which is
impossible by Lemma \ref{E6}$(x)$.

Assume next that $\sigma=\gamma\varrho,$ where $\varrho$ is a nontrivial field
automorphism and $\gamma$ is either trivial or a diagonal automorphism.  Let $\ell=\ell_{12f}(p)$ be a primitive prime divisor of $p^{12f}-1$ and let $s$ be a
semisimple element of order $\ell$ in $K\cong {\rm F}_4(q)$ where $K\leq (H^*)'$ is chosen so that it is
fixed under the involutory graph $\tau$ (see \cite[Proposition~4.9.2]{Gorenstein}).  Now $|C_{(\mathcal{G}^*)^{F^*}}(s)|=\Phi_{12}\Phi_3$ and $\mathcal{G}^F$ possesses an irreducible
character $\chi_s$ with degree
$$|(\mathcal{G}^*)^{F^*}:C_{(\mathcal{G}^*)^{F^*}}(s)|_{p'}=|{\rm E}_6(q)_{ad}|_{p'}/\Phi_{12}\Phi_3$$ which is
trivial on the center as $s\in (H^*)'$ so $\chi_s\in\Irr(H).$  Let $\varrho^*$ be an automorphism of
$\mathcal{G}^*$ induced from the field automorphism $\varrho.$ Then
$(\mathcal{G}^*)^{\varrho^*}\cong {\rm E}_6(p^b)_{ad},$ where $1\leq b<f,$ so $\ell=\ell_{12f}(p)$ does
not divide $|(\mathcal{G}^*)^{\varrho^*}|$. By \cite[Lemma~2.5(ii)]{Dolfi}, 
$\chi_s$ is not $\varrho$-invariant. As $\chi_s$ is invariant under the diagonal automorphisms, it is not $\varrho$-invariant and hence some proper multiple of $\chi_s(1)$ is a
degree of $H,$ which is impossible.

Now assume that $\sigma=\gamma\varrho\tau,$ where $\gamma$ is either trivial or a diagonal
automorphism, $\varrho$ is a nontrivial field automorphism and $\tau$ is a graph automorphism. As the field and graph automorphisms commute by \cite[Theorem~2.5.12]{Gorenstein}, we
can write $\sigma=\gamma\tau\varrho.$ Choose  $s\in (H^*)'$ as in previous
paragraph with $|s|=\ell_{12}(q).$ Since $s$ in $\tau$-invariant, \cite[Corollary~2.5]{Navarro}
yields that $\chi_s^\tau=\chi_{\tau(s)}=\chi_s.$ As $\chi_s$ is also $\gamma$-invariant, by applying
\cite[Corollary~2.5]{Navarro} again $\chi_s$ is $\gamma\tau$-invariant. Thus
$\chi_s^\sigma=\chi_s^{\gamma\tau\varrho}=\chi_s^\varrho\neq \chi_s.$  Therefore $\chi_s$ is not
$\sigma$-invariant, which leads to a contradiction.

Finally assume that $\sigma=\gamma\tau,$ where $\gamma$ is either trivial or a diagonal
automorphism and $\tau$ is a graph automorphism. Let $\ell=\ell_{9f}(p)$ be a primitive prime
divisor of $p^{9f}-1$ and let $s$ be a semisimple element of order $\ell$ in  $ (H^*)'.$ Let $T$ be
a maximal torus of $H^*$ containing $s.$ It is well known that $T$ is uniquely determined up to
conjugation in $H^*.$  Now $|C_{(\mathcal{G}^*)^{F^*}}(s)|=\Phi_{9}$ and $\mathcal{G}^F$
possesses an irreducible character $\chi_s$ with degree
$$|(\mathcal{G}^*)^{F^*}:C_{(\mathcal{G}^*)^{F^*}}(s)|_{p'}=|{\rm E}_6(q)_{ad}|_{p'}/\Phi_{9},$$ which is
trivial on the center as $s\in (H^*)'$ so $\chi_s\in\Irr(H).$  As $\chi_s$ is invariant under the
diagonal automorphisms,  it suffices to show that $\chi_s$ is not $\tau$-invariant: if so, then it is not
$\sigma$-invariant and hence some proper multiple of $\chi_s(1)$ is a degree of $H,$ which is
impossible. We know that $\Centralizer_{H^*}(\tau)\cong \textrm{F}_4(q)$ and thus $\ell\nmid
|\Centralizer_{H^*}(\tau)|.$ Consider the $H^*$-conjugacy class $(s).$ Since $|s|=\ell,$  $H^*$-conjugacy class $(s)$ is not $\tau$-invariant. By \cite[Corollary~2.5]{Navarro},
$\chi_s$ is not $\tau$-invariant.

The arguments for the remaining cases are similar. 
\end{proof}

We now complete the proof of the main theorem.
\begin{proof}[\textbf{Proof of Theorem \ref{th:main}}]
Let $G$ be a finite group and let $H$ be a finite simple exceptional group of Lie type.
In view of \cite{Hupp,Tong,HT1,HT2,HT,Wake}, we assume that $H$ is one of the simple groups listed in \eqref{eqn1}. By Propositions \ref{prop:step3} and \ref{prop:step4},  $G'\cong H$ and  $G\cong G' \times C_G(G')$. It follows that $C_G(G')\cong G/G'$ is abelian. Hence $G\cong G'\times A$, where $A= C_G(G')$ is an abelian group. This completes the proof of the theorem. 
\end{proof}

\subsection*{Acknowledgment}
The author is grateful to  the reviewer and the editor for numerous suggestions and corrections that improve the exposition of the paper. The author also thanks Frank L\"{u}beck for his help during the preparation of this paper.

\end{document}